\documentclass{amsart}
\usepackage{microtype}
\usepackage{amssymb}
\usepackage{mathtools}
\usepackage{hyperref}
\usepackage{tikz-cd}
\usepackage{mathrsfs}
\usepackage{pinlabel}
\usepackage{booktabs}

\usetikzlibrary{decorations.markings}
\tikzset{negated/.style={
        decoration={markings,
            mark= at position 0.5 with {
                \node[transform shape] (tempnode) {\tiny$\backslash$};
            }
        },
        postaction={decorate}
    }
}

\usepackage[capitalise, noabbrev]{cleveref}

\usepackage{enumitem}
\setlist{labelindent=0.5\parindent, leftmargin=*}

\usepackage[old]{old-arrows}
\usepackage[only,mapsfromchar]{stmaryrd}
\newcommand{\ract}{\varrightarrow\mapsfromchar}

\DeclareFontFamily{U} {MnSymbolC}{}

\DeclareFontShape{U}{MnSymbolC}{m}{n}{
  <-6> MnSymbolC5
  <6-7> MnSymbolC6
  <7-8> MnSymbolC7
  <8-9> MnSymbolC8
  <9-10> MnSymbolC9
  <10-12> MnSymbolC10
  <12-> MnSymbolC12}{}
\DeclareFontShape{U}{MnSymbolC}{b}{n}{
  <-6> MnSymbolC-Bold5
  <6-7> MnSymbolC-Bold6
  <7-8> MnSymbolC-Bold7
  <8-9> MnSymbolC-Bold8
  <9-10> MnSymbolC-Bold9
  <10-12> MnSymbolC-Bold10
  <12-> MnSymbolC-Bold12}{}

\DeclareSymbolFont{MnSyC} {U} {MnSymbolC}{m}{n}
\DeclareMathSymbol{\medstar}{\mathbin}{MnSyC}{130}

\DeclareMathOperator{\fun}{Fun}

\DeclareMathOperator*{\colim}{colim}

\DeclareMathOperator{\seg}{Seg}
\newcommand{\spshf}[1]{\mathrm{s}\widehat{\mathbf{#1}}}

\newcommand{\freecyc}{\mathsf{C}}
\newcommand{\freemod}{\mathsf{M}}

\newcommand{\actrm}{\mathrm{act}}
\newcommand{\intrm}{\mathrm{int}}
\newcommand{\elrm}{\mathrm{el}}
\newcommand{\oprm}{\mathrm{op}}

\newcommand{\finset}{\mathbf{FinSet}}
\newcommand{\wproperad}{\mathbf{WPpd}}
\newcommand{\properad}{\mathbf{Ppd}}
\newcommand{\modular}{\mathbf{ModOp}}
\newcommand{\cyc}{\mathbf{Cyc}}
\newcommand{\augcyc}{\mathbf{Cyc}^+}
\newcommand{\diop}{\mathbf{DiOp}}
\newcommand{\operad}{\mathbf{Opd}}
\newcommand{\set}{\mathbf{Set}}
\newcommand{\cat}{\mathbf{Cat}}
\newcommand{\gcat}{\mathbf{U}}
\newcommand{\gcatemb}{\gcat_{\mathrm{int}}}
\newcommand{\gcatact}{\gcat_{\actrm}}
\newcommand{\gcatcyc}{\gcat_{\mathrm{cyc}}}
\newcommand{\gcatnought}{\gcat_{0}}
\newcommand{\opshf}{\mathfrak{o}}
\newcommand{\ocat}{\mathbf{O}}
\newcommand{\pgcat}{\mathbf{G}}
\newcommand{\pgcatsc}{\mathbf{O}_{0}}
\newcommand{\dendcat}{\mathbf{\Omega}}
\newcommand{\simpcat}{\mathbf{\Delta}}
\DeclareMathOperator{\id}{id}
\DeclareMathOperator{\nbhd}{nb}
\newcommand{\emb}{\operatorname{Emb}}
\newcommand{\inp}{\operatorname{in}}
\newcommand{\out}{\operatorname{out}}
\newcommand{\ssub}{\operatorname{sSb}}
\newcommand{\rat}{\rightarrowtail}
\newcommand{\vertfunc}{\textup{V}}

\newcommand{\mydef}[1]{\textbf{#1}}

\newcommand{\CC}{\mathbf{C}}
\newcommand{\DD}{\mathbf{D}}

\newcommand{\fstarop}{\mathbf{F}_\ast^\oprm}
\newcommand{\finsetstarop}{\finset_\ast^\oprm}
\newcommand{\lr}[1]{\langle #1 \rangle}

\newcommand{\PP}{\mathbf{P}}

\tikzset{
  act /.tip = >|
}

\newcommand{\hphi}{\hat\varphi}

\newtheorem{theorem}{Theorem}[section]
\newtheorem{lemma}[theorem]{Lemma}
\newtheorem{proposition}[theorem]{Proposition}
\newtheorem{corollary}[theorem]{Corollary}
\theoremstyle{definition}
\newtheorem{definition}[theorem]{Definition}
\newtheorem{construction}[theorem]{Construction}
\newtheorem{remark}[theorem]{Remark}
\newtheorem{example}[theorem]{Example}
\newtheorem{question}[theorem]{Question}
\newtheorem{slogan}[theorem]{Slogan}

\title{Segal conditions for generalized operads}

\author{Philip Hackney}
\address{Department of Mathematics, University of Louisiana at Lafayette, USA}
\email{philip@phck.net} 
\urladdr{http://phck.net}
\thanks{This work was supported by a grant from the Simons Foundation (\#850849)}


\begin{document}
\begin{abstract}
This note is an introduction to several generalizations of the \emph{dendroidal sets} of Moerdijk--Weiss.
Dendroidal sets are presheaves on a category of rooted trees, and here we consider indexing categories whose objects are other kinds of graphs with loose ends.
We examine the \emph{Segal condition} for presheaves on these graph categories, which is one way to identify those presheaves that are a certain kind of generalized operad (for instance wheeled properad or modular operad).
Several free / forgetful adjunctions between different kinds of generalized operads can be realized at the presheaf level using only the left Kan extension / restriction adjunction along a functor of graph categories.
These considerations also have bearing on homotopy-coherent versions of generalized operads, and we include some questions along these lines.
\end{abstract}

\maketitle

\section{Introduction}
A simplicial set $X$ is said to \emph{satisfy the Segal condition} if the maps
\[ \begin{tikzcd}[row sep=0, column sep=large]
X_2 \rar["{d_2, d_0}"] & X_1 \times_{X_0} X_1 \\ 
X_3 \rar["{d_2 d_3, d_0 d_3, d_0 d_0}"] & X_1 \times_{X_0} X_1 \times_{X_0} X_1 \\[-0.2cm] 
\vdots & \vdots \\
X_n \rar & X_1 \times_{X_0} X_1 \times_{X_0} \dots  \times_{X_0} X_1
\end{tikzcd} \]
are bijections for every $n \geq 2$.
Such simplicial sets are the same thing as \emph{categories}: $X_0$ is the set of objects, $X_1$ is the set of morphisms, and $d_1 \colon X_2 \to X_1$ provides a composition operation $X_1 \times_{X_0} X_1 \to X_1$ (see \cite[Proposition 4.1]{Grothendieck:TCTGA3}).
This situation can be fruitfully generalized to simplicial spaces by replacing bijections with weak equivalences, leading toward a model for $(\infty,1)$-categories due to Rezk \cite{Rezk:MHTHT}.
This approach was extended by Cisinski and Moerdijk to provide a model for $(\infty,1)$-operads in \cite{CisinskiMoerdijk:DSSIO}.

In this note, we discuss further extensions the Segal condition (in both the weak and strict sense) from categories and (colored) operads to other sorts of operadic structures.
Though generalizations of operads have appeared in many guises \cite{YauJohnson:FPAM,KaufmannWard:FC,BataninMarkl:OCDDC,BataninBerger:HTAPM,Getzler:OR}, we take here an example-driven approach.
Our focus will be on colored operads, dioperads/symmetric polycategories \cite{Gan:KDD,Garner:PPDL}, cyclic operads \cite{GetzlerKapranov:COCH}, properads \cite{Vallette:KDP,Duncan:TQC}, wheeled properads \cite{MarklMerkulovShadrin}, modular operads \cite{GetzlerKapranov:MO}, and props \cite{MacLane:CA}.
Note that the generalizations of operads in the above sources are usually focused on the monochrome case (or at least the case where morphisms fix the color set), while we consider the `category-like' version where operations are only composable if they have compatible colors/types.
This setting is useful for applications, see for instance \cite{FBSD:ComplexSystem,Spivak:OWD,VagnerSpivakLerman:AODSOWD}, and is formally clarifying.

Given a kind of generalized operad, the Segal condition is then formulated using some category of graphs; for categories, we can regard the simplicial category $\simpcat$ as a category of linear trees, and, for operads, the Moerdijk--Weiss dendroidal category $\dendcat$ is a category of rooted trees.
We will present appropriate graph categories for each of the kinds of operadic structures we are focusing on (with the exception of props; see \cref{question disconnected}).
Each such graph category has a natural notion of Segal presheaf.
Additionally, these graph categories are generalized Reedy categories in the sense of \cite{bm_reedy}, which has model-categorical implications for the homotopy-coherent situation.

This note is, in part, an expansion of the material from my talk ``Interactions between different kinds of generalized operads'' in the AMS \emph{Special Session on Higher Structures in Topology, Geometry and Physics} in March 2022, and in part a distillation of the core ideas from \cite{Hackney:CGOS}.
That paper introduced a new description of morphisms for several graph categories (most of which were introduced earlier by the author in joint work with Robertson and Yau), and was structured to allow efficient comparison with pre-existing definitions.
In contrast, the present paper is designed as an overview for someone coming to the subject for the first time. 
Our goals are to give basic definitions of graph categories in elementary terms, explain the Segal condition for presheaves, and discuss the connections between the directed and undirected contexts.
Pointers are provided to the relevant literature, and we highlight several interesting questions near the end.
For particular graph categories, there are other introductory accounts, each having a different focus: \cite{Weiss:FODS}, \cite{Moerdijk:LDS}, and the first part of \cite{HeutsMoerdijk:SDHT} are good starting points for the dendroidal theory (the book is a comprehensive account of the theory), \cite{MR3792530} is focused on $\infty$-properads, while the forthcoming \cite{BorghiRobertson:MIOGT} concerns modular $\infty$-operads.

\subsection{What is a generalized operad?}
This is a question we don't wish to answer too directly or precisely.
As a first approximation, `generalized operad' should include all of the mathematical objects in MSC2020 code \texttt{18M85},\footnote{\textbf{18M85} Polycategories/dioperads, properads, PROPs, cyclic operads, modular operads}
as well as some closely related objects.
At its core, a generalized operad should consist of a collection of \emph{operations}, each of which comes equipped with some boundary information (for instance coloring by some set, an `input' or `output' tag, or so on), as well as specified ways of connecting (compatible) boundaries of operations to form new operations.
See \cref{fig operations}.

\begin{figure}
\includegraphics{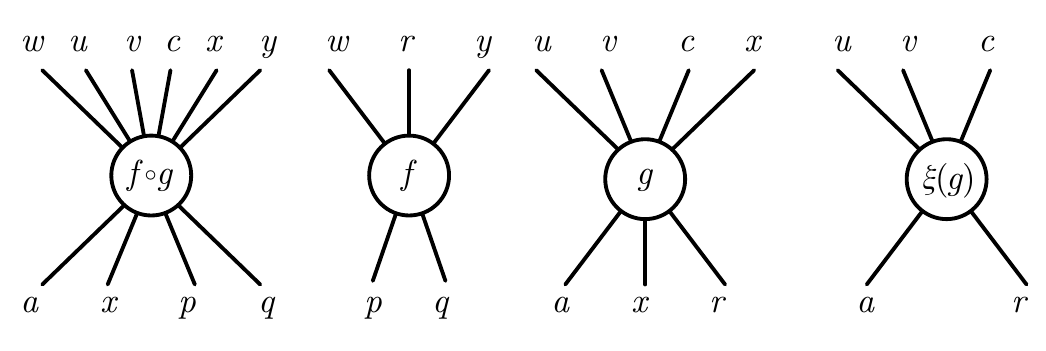}
\caption{Operations $f\colon w,r,y \to p,q$ and $g\colon u,v,c,x \to a,x,r$ along with composition $f\circ_r g$ and contraction $\xi_x(g)$.}
\label{fig operations}
\end{figure}

In a moment we'll give a brief list of several of the structures that will be discussed, but first let's note some general features.
Operad-like structures will always be \emph{colored} or \emph{typed} (by a set or involutive set), and, as with functors between categories, maps between them are not expected be the identity on color sets.
Each color will have an associated \emph{identity} operation.
We will always be discussing the \emph{symmetric} version of concepts, as for many structures (e.g., properads, modular operads) the non-symmetric version is quite different. 
Some structures will be \emph{directed} in the sense that there is a strong distinction between inputs and outputs of operations, while others will be \emph{undirected} and may be considered to just have `puts' (or `ports').

In the following list of examples, we'll begin with directed operadic structures, roughly in order of increasing complexity, before turning to undirected operadic structures.
Several references to precise definitions in the literature will be given just below the list.

\begin{description}[wide=0pt]
\item[Category] Morphisms in a category are of the form $f\colon c\to d$.
Given two morphisms $f$ and $g$, there are at most two ways to compose the morphisms: $g\circ f$ and $f\circ g$, provided they exist.
\item[Operad] An operad is a generalization of a category, where operations now can have multiple (or no) inputs, they are of the form $f\colon c_1, \dots, c_n \to d$.
There are compositions
\[
	P(c_1, \dots, c_n; d) \times P(d_1, \dots, d_m; c_i) \xrightarrow{\circ_i} P(c_1, \dots, c_{i-1}, d_1, \dots, d_m, c_{i+1}, \dots, c_n; d)
\]
which may give many ways to compose two given operations.
\item[Dioperad] A dioperad, also called a polycategory, expands operads to allow multiple outputs as well. 
To compose two operations, one chooses a single input of the first which agrees with a single input of the second, and connects them.
The relevant structure map when $c_i = b_j$ is $\vphantom{\circ}_i\circ_j$ below.
\[
\begin{tikzcd}[sep=small]
	P(c_1, \dots, c_n; d_1, \dots, d_m) \times P(a_1, \dots, a_k;b_1, \dots, b_\ell) \dar{ \vphantom{\circ}_i\circ_j} \\
	 P(c_1, \dots, c_{i-1}, a_1, \dots, a_k, c_{i+1}, \dots, c_n; b_1, \dots, b_{j-1}, d_1, \dots, d_m , b_{j+1}, \dots, b_\ell )
\end{tikzcd}
\]
\item[Properads] Arguably this is the most complicated structure on this list. 
The underlying operation sets are the same as those of dioperads, but now one can connect one \emph{or more} inputs of the first operation to the same number of outputs of the second operation.
\item[Props] These are like properads, but now one can connect \emph{zero or more} inputs of the first operation to the same number of inputs of the second operation. 
More simply, props are just those symmetric monoidal categories whose set of objects is a free monoid under tensor.
The set of colors is the free generating set for the monoid of objects.
\item[Wheeled properads]
A wheeled properad has additional \emph{contractions} that act on a single operation. They take the form
\[
	P(c_1, \dots, c_n; d_1, \dots, d_m) \to P(c_1, \dots, \hat c_i, \dots, c_n; d_1, \dots, \hat d_j, \dots, d_m)
\]
where $c_i = d_j$.
The properadic compositions in a wheeled properad all decompose into simpler ones.
Namely, a composition connecting $k > 0$ inputs of one operation to $k$ outputs of another composes as a dioperadic composition, followed by $k-1$ contractions.
Wheeled properads could instead be called `wheeled dioperads.'
\item[Involutive category] A \emph{dagger category} is a category $C$ equipped with an identity-on-objects functor $C \to C^\oprm$ so that $C \to C^\oprm \to (C^{\oprm})^{\oprm} = C$ is $\id_C$. 
That is, every $f \colon c\to d$ has an associated $f^\dagger \colon d \to c$ (with $(f^\dagger)^\dagger = f$), and this allows you to exchange the role of input and output of morphisms.
More generally, an \emph{involutive category} relaxes the requirement that $C \to C^\oprm$ is the identity-on-objects, so every $f\colon c \to d$ has an associated $f^\dagger \colon d^\dagger \to c^\dagger$.
\item[Cyclic operad] A cyclic operad $P$ is an operad equipped with an involution $\dagger$ on its set of colors, along with additional structure maps
\[
	P(c_1, \dots, c_n; c_0) \xrightarrow{\cong} P(c_2, \dots, c_n, c_0^\dagger; c_1^\dagger)
\]
that combine into a $\mathbb{Z}/(n+1)\mathbb{Z}$-action on operations of arity $n$. 
This allows one to exchange the roles of inputs and outputs.
If the underlying operad $P$ is just a category, this notion coincides with that of being an involutive category.
\item[Augmented cyclic operad] If one can exchange the roles of inputs and outputs, is there any real distinction between the two?
Operations in an augmented cyclic operad dispense with this, and regard inputs and outputs on entirely the same footing.
Operation sets take the form $P(c_1, \dots, c_n)$ (for $n\geq 0$), and compositions take the form
\[
\begin{tikzcd}[sep=small]
	P(c_1, \dots, c_n) \times P(d_1, \dots, d_m) \dar{ \vphantom{\circ}_i\circ_j} \\
	 P(c_1, \dots, c_{i-1}, d_{j+1}, \dots, d_{m}, d_0, \dots, d_{j-1}, c_{i+1}, \dots, c_n)
\end{tikzcd}
\]
whenever $c_i = d_j^\dagger$.
The real difference is that in a cyclic operad you cannot compose two operations which only have outputs but no inputs, while in an augmented cyclic operad this is possible, and the result will land in the set $P(\,\,)$ which has no analogue for cyclic operads.
\item[Modular operads]
These are augmented cyclic operads, together with contraction operations
\[
	P(c_1, \dots, c_n) \to P(c_1, \dots, \hat c_i, \dots, \hat c_j, \dots, c_n)
\]
for $i < j$ with $c_i = c_j^\dagger$.
\end{description}

Definitions in these terms for the directed structures may be found, for instance, in Chapter 11 of \cite{YauJohnson:FPAM}; see also \cite[\S2.2]{GarnerHirschowitz:SMAF}, \cite[\S6.1]{Duncan:TQC}, and \cite{HackneyRobertson:OCP}.
For cyclic operads see \cite[Definition 7.4]{Shulman:2Chu} and \cite[Definition 3.3]{ChengGurskiRiehl:CMMAM}, for augmented cyclic operads see \cite[Definition 2.3]{DrummondColeHackney:DKHTCO}, and for modular operads see \cite[Definition 1.24]{Raynor:DLCSM}.
But please don't go and look all of this up: according to the contents of \cref{sec segal}, many of these operadic structures could instead be defined as \emph{Segal presheaves} over the graph categories we introduce here, just as in the case of categories.
There are of course a number of other interesting and useful operadic structures that do not appear on the above list, some of which can even fit into the frameworks we lay out below.

\section{Graphs}\label{sec: graphs}

In operadic contexts, the correct notion of graphs to use is one with `loose ends.' 
There are a number of different combinatorial models for this, and we now recall one that has a simple associated notion of morphism, due to Joyal and Kock \cite[\S3]{JoyalKock:FGNTCSM}. 
\begin{definition}[Undirected graphs]\label{def jk graphs}
An \mydef{undirected graph} $G$ is a diagram of finite sets
\[\begin{tikzcd}[column sep=small] A_G \arrow[loop left,"\dagger"] & \lar[hook'] D_G \rar{t} & V_G \end{tikzcd}\]
with $\dagger$ a fixpoint-free involution and $D_G \to A_G$ a monomorphism.
We call $A_G$ the set of \mydef{arcs} and $V_G$ the set of \mydef{vertices}.
\end{definition}
A corresponding encoding of \emph{directed} graph with loose ends was introduced by Kock as Definition 1.1.1 of \cite{Kock:GHP}.
(Graphs of this type have recently been used to give new perspectives on Petri nets \cite{BaezGenoveseMasterShulman,Kock:WGPNP}.)

\begin{definition}[Directed graphs]\label{def dir graph}
A \mydef{directed graph} $G$ is a diagram of finite sets of the form
\[
\begin{tikzcd}[column sep=small]
	E_G & I_G \rar\lar[hook'] & V_G & O_G \rar[hook] \lar & E_G,
\end{tikzcd}
\]
where the two end maps are monomorphisms.
We call elements of $E_G$ the \mydef{edges} of the graph.
\end{definition}

Every directed graph determines a canonical associated undirected graph, with $A_G \coloneqq E_G \amalg E_G$ together with the swap involution, and $D_G \coloneqq I_G \amalg O_G$.
We will give important examples in a moment, but for now let us introduce some subsidiary notions.

\begin{definition} Suppose $G$ is an undirected graph.
\begin{itemize}
\item
If $v\in V_G$ is a vertex, then $\nbhd(v) \coloneqq t^{-1}(v) \subseteq D_G$ is the \mydef{neighborhood} of $v$.
	\item 
The \mydef{boundary} of $G$ is the set $\eth(G) = A_G \backslash D_G$.
\item 
The set of \mydef{edges} is the set of $\dagger$-orbits.
\[
	E_G = \{ [a,a^\dagger] \mid a \in A_G\} \cong A_G / {\sim}
\]
\item
An edge $[a,a^\dagger]$ is an \mydef{internal edge} if neither of $a,a^\dagger$ are elements of $\eth(G)$. 
\end{itemize}
Suppose $G$ is a directed graph.
\begin{itemize}
\item If $v\in V_G$ is a vertex, then $\inp(v) \subseteq I_G$ is the preimage of $v$ under $I_G \to V_G$, called the set of \mydef{inputs of} $v$. Likewise, $v$ has a set of \mydef{outputs} $\out(v) \subseteq O_G$.
\item The set of \mydef{input edges of} $G$ is defined to be $\inp(G) \coloneqq E_G \backslash O_G$ and the set of \mydef{output edges of} $G$ is $\out(G) \coloneqq E_G \backslash I_G$. 
All other edges are \textbf{internal edges}.
\end{itemize}
\end{definition}

Each undirected graph $G$ determines a topological pair $\mathscr{X}_G = (X_G, V_G)$, called the \mydef{geometric realization}. Each arc in $A_G$ gets an open interval, each element of $D_G$ gets a half-open interval, each vertex gets a point, and these are glued together in the manner described by the diagram for $G$ (e.g., if $t\in (0,1)$ is in the interval associated to $a$, then it will be identified with $1-t\in (0,1)$ in the interval associated to $a^\dagger$).
We mention this mostly in that we will draw pictures of our graphs, and we will frequently exaggerate the size of the vertices. 
We freely use the terms `connected' and `simply-connected' to refer to properties of the topological space $X_G$ (though these can also be described combinatorially). 

\begin{example}\label{example undirected graph}
\Cref{fig graph example} gives a picture of an undirected graph $G$ with loose ends.
The arc set $A_G = \{n,n^\dagger \mid 1 \leq n \leq 9\}$ has eighteen elements, together with the evident involution swapping $n$ and $n^\dagger$, while the vertex set is $V_G = \{ u, v, w, x \}$.
The set $D_G$ is given below
\[
D_G = 
\{ 
\underbrace{1^\dagger}_{\nbhd(u)}, 
\underbrace{3^\dagger, 4, 6^\dagger, 8, 7}_{\nbhd(v)},
\underbrace{4^\dagger, 5, 5^\dagger, 6}_{\nbhd(w)},
\underbrace{8^\dagger, 9}_{\nbhd(x)}
 \}
\]
and the indicated partition specifies the function $t\colon D_G \to V_G$.
The boundary set is $\eth(G) = \{1, 2, 2^\dagger, 3, 7^\dagger, 9^\dagger \}$, which has six elements.
\end{example}

\begin{figure}
\labellist
\small\hair 2pt
 \pinlabel {$1$} [ ] at 24 118
 \pinlabel {$1^\dagger$} [ ] at 59 94
 \pinlabel {$2$} [ ] at 72 138
 \pinlabel {$2^\dagger$} [ ] at 130 140
 \pinlabel {$3$} [ ] at 112 87
 \pinlabel {$3^\dagger$} [ ] at 143 89
 \pinlabel {$4$} [ ] at 157 101
 \pinlabel {$4^\dagger$} [ ] at 193 146
 \pinlabel {$5$} [ ] at 221 160
 \pinlabel {$5^\dagger$} [ ] at 242 142
 \pinlabel {$6$} [ ] at 227 114
 \pinlabel {$6^\dagger$} [ ] at 187 92
 \pinlabel {$7$} [ ] at 148 59
 \pinlabel {$7^\dagger$} [ ] at 138 33
 \pinlabel {$8$} [ ] at 174 56
 \pinlabel {$8^\dagger$} [ ] at 196 33
 \pinlabel {$9$} [ ] at 241 39
 \pinlabel {$9^\dagger$} [ ] at 260 60
 \pinlabel {$u$} [ ] at 64.5 68.5
 \pinlabel {$v$} [ ] at 164 77
 \pinlabel {$w$} [ ] at 215 136
 \pinlabel {$x$} [ ] at 221 26
\endlabellist
\centering
\includegraphics[scale=0.5]{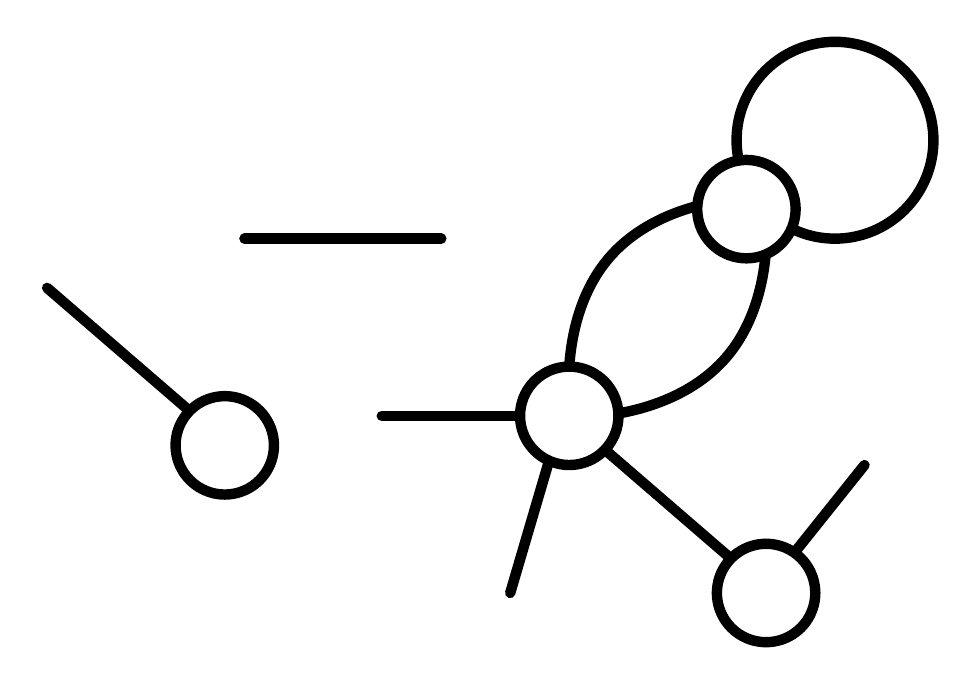}
\caption{An undirected graph from \cref{example undirected graph}}
\label{fig graph example}
\end{figure}
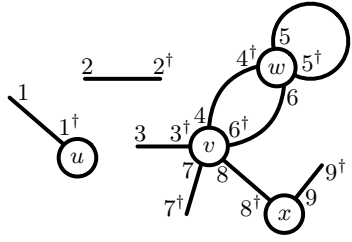

\begin{remark}
A different combinatorial definition of (undirected) graph with loose ends is common in the literature (see, e.g., \cite[6.6.1]{KontsevichManin:GWCQCEG} or \cite[Definition 1.1]{BehrendManin:SSMGWI}).
This is simply a diagram of finite sets
\[\begin{tikzcd}[column sep=small] F \arrow[loop left] \rar & V \end{tikzcd}\]
where the self-map of $F$ is an involution.
The fixed points of the involution represent the loose ends of the graph, and the free orbits represent edges between vertices.
Given such a graph, we can form a graph as in \cref{def jk graphs} by setting $D \coloneqq F$, and to give $A$ we double up the fixed points of the involution on $F$ to turn it into a free involution.
However, not all undirected graphs arise from this process: we never see graphs which have edges that float free. 
See \cite[\S15.3]{BataninBerger:HTAPM} for more details.
\end{remark}

\begin{figure}
\includegraphics[scale=0.3]{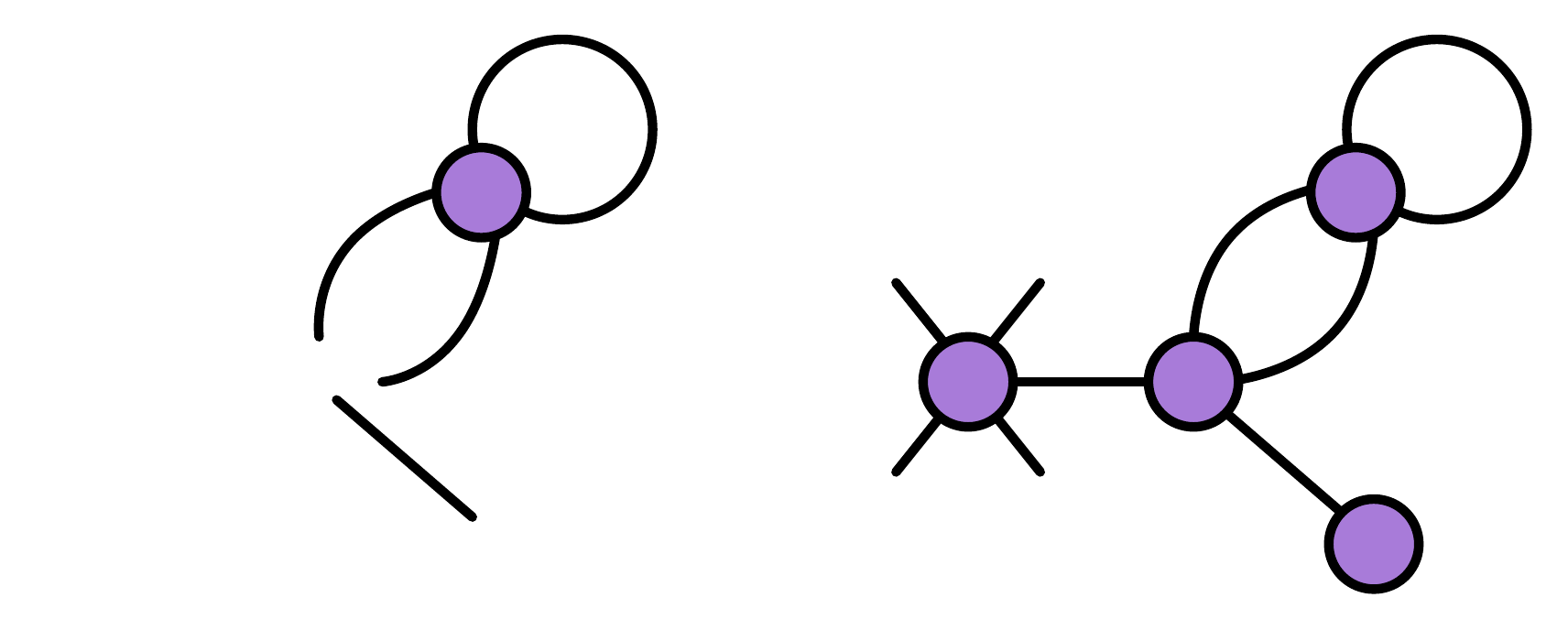}
\caption{Subgraph determined by one vertex and four edges}\label{fig subgraph}
\end{figure}

\begin{definition}[Subgraph]\label{def subgraph}
Suppose $G$ is a graph, either directed or undirected.
A \mydef{subgraph} of $G$ is determined by a pair of subsets $E_H \subseteq E_G$ and $V_H \subseteq V_G$ subject to a single condition.
\begin{itemize}
\item Suppose $G$ is undirected, and let $A_H \subseteq A_G$ denote the set of all arcs appearing in the edges in $E_H$; this is naturally an involutive subset of $A_G$.
Further, let $D_H \coloneqq t^{-1}(V_H) \subseteq D_G$ be the set of those elements mapping into $V_H$.
Then $(E_H,V_H)$ determines a subgraph just when $D_H \to D_G \to A_G$ lands in the subset $A_H$.
\begin{equation}\label{eq undirected subgraph} \begin{tikzcd}
A_H\dar[hook]  \arrow[loop left,"\dagger"] & \lar[hook', dashed] D_H\dar[hook]  \rar & V_H\dar[hook]  \\
A_G \arrow[loop left,"\dagger"] & \lar[hook'] D_G \rar & V_G
\end{tikzcd} \end{equation}
\item Suppose $G$ is directed, and let $I_H \subseteq I_G$ and $O_H \subseteq O_G$ be the preimages of the subset $V_H$ along the functions $I_G \to V_G$ and $V_G \leftarrow O_G$. 
Then $(E_H,V_H)$ determines just when the dashed arrows in the following diagram exist.
\begin{equation}\label{eq directed subgraph}
\begin{tikzcd}
E_H  \dar[hook] & I_H \rar\lar[hook', dashed]  \dar[hook]  & V_H  \dar[hook]& O_H \rar[hook, dashed] \lar  \dar[hook] & E_H \dar[hook] \\
E_G & I_G \rar\lar[hook'] & V_G & O_G \rar[hook] \lar & E_G	
\end{tikzcd}
\end{equation}
\end{itemize}
The top row in \eqref{eq undirected subgraph} or \eqref{eq directed subgraph} exhibits $H$ as an undirected or directed graph.
\end{definition}
Of course we could write these requirements in other ways, such as declaring that a subgraph is a triple or quadruple of subsets satisfying some properties.
But we prefer to emphasize that subgraphs are always determined by a set of edges and a set of vertices.
Notice also that a subgraph of a directed graph is nothing but a subgraph of its underlying undirected graph. 

We now describe two kinds of important graphs, which we term \mydef{elementary}.
For a non-negative integer $n$, we will write $\mathbf{n}$ for the set $\{1,2,\dots,n\}$.

\begin{example}[Edge]\label{ex: edge}
The canonical \mydef{undirected edge}, here denoted $\updownarrow$, is 
\[\begin{tikzcd}[column sep=small] \mathbf{2} \arrow[loop left,"\dagger"] & \lar[hook'] \mathbf{0} \rar{t} & \mathbf{0}. \end{tikzcd}\]
That is, this graph has no vertices, and all arcs are in the boundary.
The canonical \mydef{directed edge}, here denoted $\downarrow$, is 
\[
\begin{tikzcd}[column sep=small]
	\mathbf{1} & \mathbf{0} \rar\lar[hook'] & \mathbf{0} & \mathbf{0} \rar[hook] \lar & \mathbf{1}.
\end{tikzcd}
\]
This graph has no vertices, and a single edge which is both an input and an output of the graph.
We also refer to any connected graph without vertices as an edge.
\end{example}

\begin{example}[Stars or corollas]\label{ex: star}
Let $n,m \geq 0$ be integers.
The $n$-star, $\medstar_n$, is the following undirected graph
\[\begin{tikzcd}[column sep=small] \mathbf{n}+\mathbf{n} \arrow[loop left,"\dagger"] & \lar[hook'] \mathbf{n} \rar{t} & \mathbf{1}, \end{tikzcd}\]
which has a single vertex with an $n$-element neighborhood, and $\eth(\medstar_n) = \mathbf{n}$.
The $n,m$-star, $\medstar_{n,m}$, is the following \emph{directed} graph
\[
\begin{tikzcd}[column sep=small]
	\mathbf{n}+\mathbf{m} & \mathbf{n} \rar\lar[hook'] & \mathbf{1} & \mathbf{m} \rar[hook] \lar & \mathbf{n}+\mathbf{m}.
\end{tikzcd}
\]
The star $\medstar_{n,m}$ has a single vertex with $n$ inputs and $m$ outputs, and $\inp(\medstar_{n,m}) = \mathbf{n}$, $\out(\medstar_{n,m}) = \mathbf{m}$.
We will also refer to any connected graph with one vertex and no internal edges as a star.
\end{example}

Every edge and star (including graphs isomorphic to these canonical forms) will be called \mydef{elementary}.
The elementary graphs are precisely those graphs with no internal edges.
Two examples of stars are given in \cref{fig star example}; we will always follow the convention that in a directed graph, the flow goes from top to bottom (in graphs with loops, this rule will apply locally at each vertex).

\begin{figure}
\includegraphics[scale=0.3]{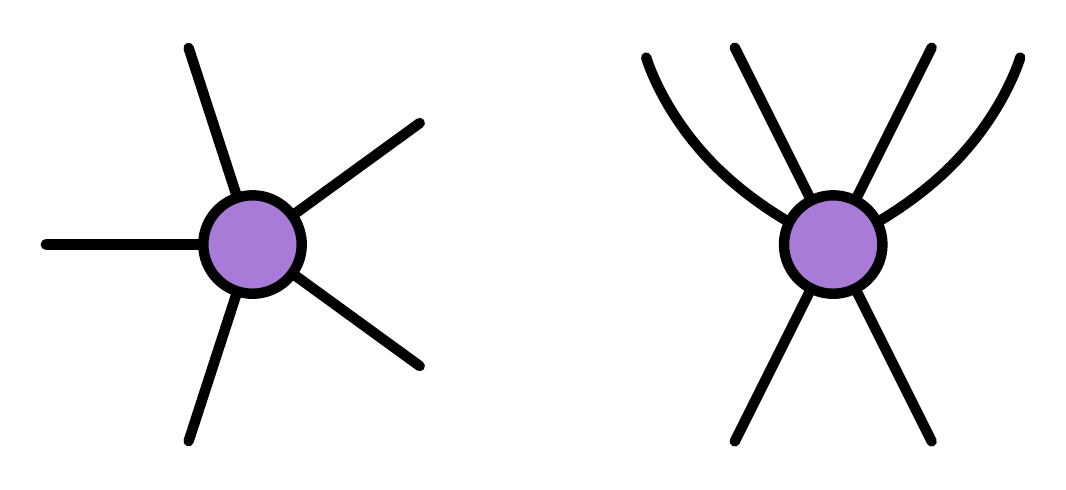}
\caption{The stars $\medstar_5$ and $\medstar_{4,2}$}\label{fig star example}
\end{figure}

If $e\in E_G$ is any edge in a graph $G$, we may consider $e$ as a subgraph as in \cref{def subgraph}.
But vertices in a graph need \emph{not} span a star-shaped subgraph in the same way. 
For example, in \cref{fig graph example}, the smallest subgraph containing the vertex $w$ has one vertex and arc set $\{4,4^\dagger,6,6^\dagger,5,5^\dagger\}$, and is not a star (the other three vertices all span star-shaped subgraphs).
If $v$ is a vertex in a graph $G$, we will write $\medstar_v$ for a star with $D_{\medstar_v} \coloneqq \nbhd(v)$ (when $G$ is undirected) or $I_{\medstar_v} \coloneqq \inp(v)$ and $O_{\medstar_v} \coloneqq \out(v)$ (when $G$ is directed).
Despite the fact that $\medstar_v$ might not be a subgraph of $G$, there is always a map (see \cref{def etale} below) $\iota_v \colon \medstar_v \to G$.

\begin{example}[Linear graphs]\label{ex linear graphs}
The linear directed graph $L_n$ with $n$ vertices (see \cref{fig linear graphs}) is presented as follows.
\[
\begin{tikzcd}
	\{0, 1, \dots, n\} & \mathbf{n} \rar{\id} \lar[hook',"-1"'] & \mathbf{n} & \mathbf{n} \rar[hook,"+0"] \lar["\id"'] & \{0, 1, \dots, n\}
\end{tikzcd}
\]
We have $\inp(L_n) = \{0\}$, $\out(L_n) = \{n\}$, while the vertex $k$ has $\inp(k) = \{k-1\}$ and $\out(k) = \{k\}$.
We also call the associated undirected graph a linear graph.
\end{example}

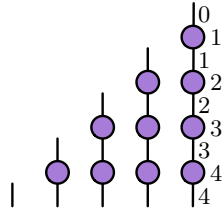
\begin{figure}
\labellist
\small\hair 2pt
\pinlabel $0$ [l] at 240 254
\pinlabel $1$ [l] at 240 197
\pinlabel $2$ [l] at 240 140
\pinlabel $3$ [l] at 240 83
\pinlabel $4$ [l] at 240 26
\pinlabel $1$ [l] at 255 225.5
\pinlabel $2$ [l] at 255 168.5
\pinlabel $3$ [l] at 255 111.5
\pinlabel $4$ [l] at 255 54.5
\endlabellist
\centering
\includegraphics[scale=0.3]{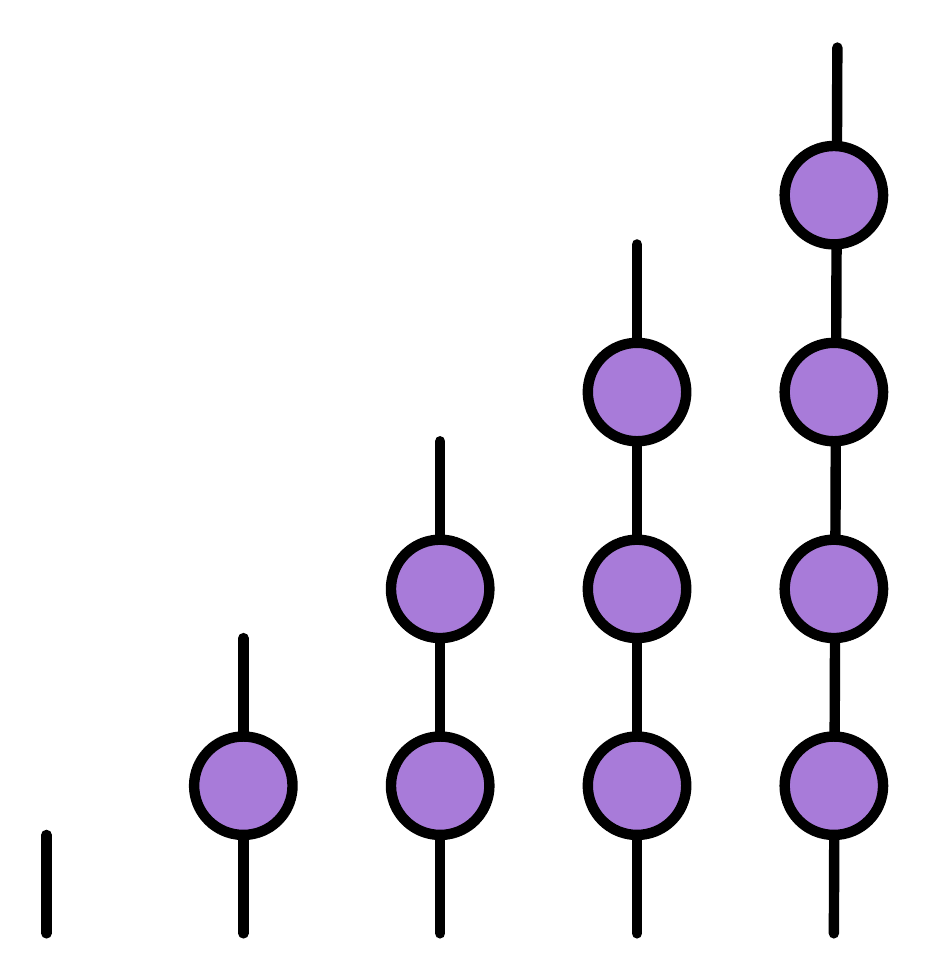}
\caption{The first five linear graphs $L_0, L_1, L_2, L_3, L_4$}
\label{fig linear graphs}
\end{figure}

\begin{definition}[paths, cycles, trees]\label{def paths trees} Let $G$ be an undirected graph.
\begin{itemize}
\item A \mydef{weak path} in $G$ is a finite alternating sequence of edges and vertices of $G$ so that if $e$ and $v$ are adjacent then some arc of $e$ appears in $\nbhd(v)$, and so that the pattern $vev$ can only appear in the path if \emph{both} arcs of $e$ are in $\nbhd(v)$.
\item A \mydef{path} in $G$ is a finite alternating sequence of arcs and vertices of $G$ so that if $av$ appears in the sequence then $t(a) = v$, and if $va$ appears in the sequence then $t(a^\dagger) = v$.
\item A \mydef{cycle} is a path of length strictly greater than one that begins and ends at the same arc or same vertex.
\item The graph $G$ is a \mydef{tree} if and only if it is connected and does not have any paths which are cycles.
\end{itemize}
\end{definition}

Each path gives a weak path by replacing the arcs with the edges they span, and each weak path has at least one associated path. 
The only ambiguity is how to handle the patterns of the form $vev$. 
In a tree, there is no real difference between paths and weak paths.
The property of being connected is equivalent to having a weak path between every pair of distinct elements of $E_G \amalg V_G$. 

\begin{definition}\label{def directed paths acyclic} Let $G$ be a directed graph.
\begin{itemize}
	\item An \mydef{undirected path} in $G$ is a weak path in the associated undirected graph.
	\item A \mydef{directed path} in $G$ is a weak path in the associated undirected graph, so that the following hold:
	\begin{itemize}
		\item If $ev$ appears in the path, then $e \in \inp(v)$.
		\item If $ve$ appears in the path, then $e \in \out(v)$.
	\end{itemize}
	\item A \mydef{directed cycle} is a directed path of length strictly greater than one that begins and ends at the same edge or vertex.
	\item A directed graph $G$ is \mydef{acyclic} if it is connected  and does not contain any directed cycles.
\end{itemize}
\end{definition}

\subsection{\'Etale maps and embeddings}
\'Etale maps are certain functions between graphs that preserve structure and preserve arity of vertices (that is, the cardinality of $\nbhd(v)$ and $\nbhd(f(v))$ coincide, and similarly for the directed case).
The following definitions appear in \cite{JoyalKock:FGNTCSM} and \cite[1.1.7]{Kock:GHP}.

\begin{definition}\label{def etale}
An \mydef{\'etale map} $G \to G'$ between undirected graphs is a triple of functions $A_G \to A_{G'}$, $D_G \to D_{G'}$, and $V_G\to V_{G'}$ so that the diagram below left commutes with right hand square a pullback.
\[ \begin{tikzcd}[sep=small]
A_G \dar & A_G\dar  \lar["\dagger"'] & \lar[hook'] D_G\dar \ar[dr, phantom, "\lrcorner" very near start]  \rar & V_G\dar  
&[+1cm] 
E_G  \dar & I_G \rar\lar[hook']  \dar \ar[dr, phantom, "\lrcorner" very near start] & V_G  \dar& O_G \rar[hook] \lar  \dar \ar[dl, phantom, "\llcorner" very near start]& E_G \dar
\\ 
A_{G'} & A_{G'} \lar["\dagger'"'] & \lar[hook'] D_{G'} \rar & V_{G'}
&[+1cm] 
E_{G'} & I_{G'} \rar\lar[hook'] & V_{G'} & O_{G'} \rar[hook] \lar & E_{G'}	
\end{tikzcd} \]
Similarly, an \'etale map $G\to G'$ between \emph{directed} graphs is a quadruple of functions $E_G \to E_{G'}$, $I_G\to I_{G'}$, $O_G\to O_{G'}$, and $V_G\to V_{G'}$ so that the diagram above right commutes and both middle squares are pullbacks.
\end{definition}

\'Etale maps induce local homeomorphisms between geometric realizations.
If $G$ is an undirected graph, then the set of \'etale maps ${\updownarrow} \to G$ is in bijection with $A_G$; likewise if $G$ is directed then $\{ {\downarrow} \to G \} \cong E_G$.
Similarly, \'etale maps from stars (modulo isomorphism in the domain) classify vertices of a graph.
Every \'etale map between directed graphs induces an \'etale map between the associated undirected graphs.
It is immediate that every subgraph in the sense of \cref{def subgraph} determines an \'etale map.
But the latter are strictly more general.

\begin{figure}
\labellist
\small\hair 2pt
\pinlabel $1$ at 2 102
\pinlabel $3$ at 160 142
\pinlabel $2$ at 50 127
\pinlabel $4$ at 115 127
\pinlabel $\rightarrow$ at 187 102
\pinlabel $1$ at 220 142
\pinlabel $2$ at 290 142
\pinlabel $\leftarrow$ at 316 102
\pinlabel $1$ at 348 142
\pinlabel $2$ at 418 142
\pinlabel $4$ at 416 65
\pinlabel $H$ at 82.5 -10
\pinlabel $G$ at 255 -10
\pinlabel $K$ at 383 -10
\endlabellist
\centering
\includegraphics[scale=0.3]{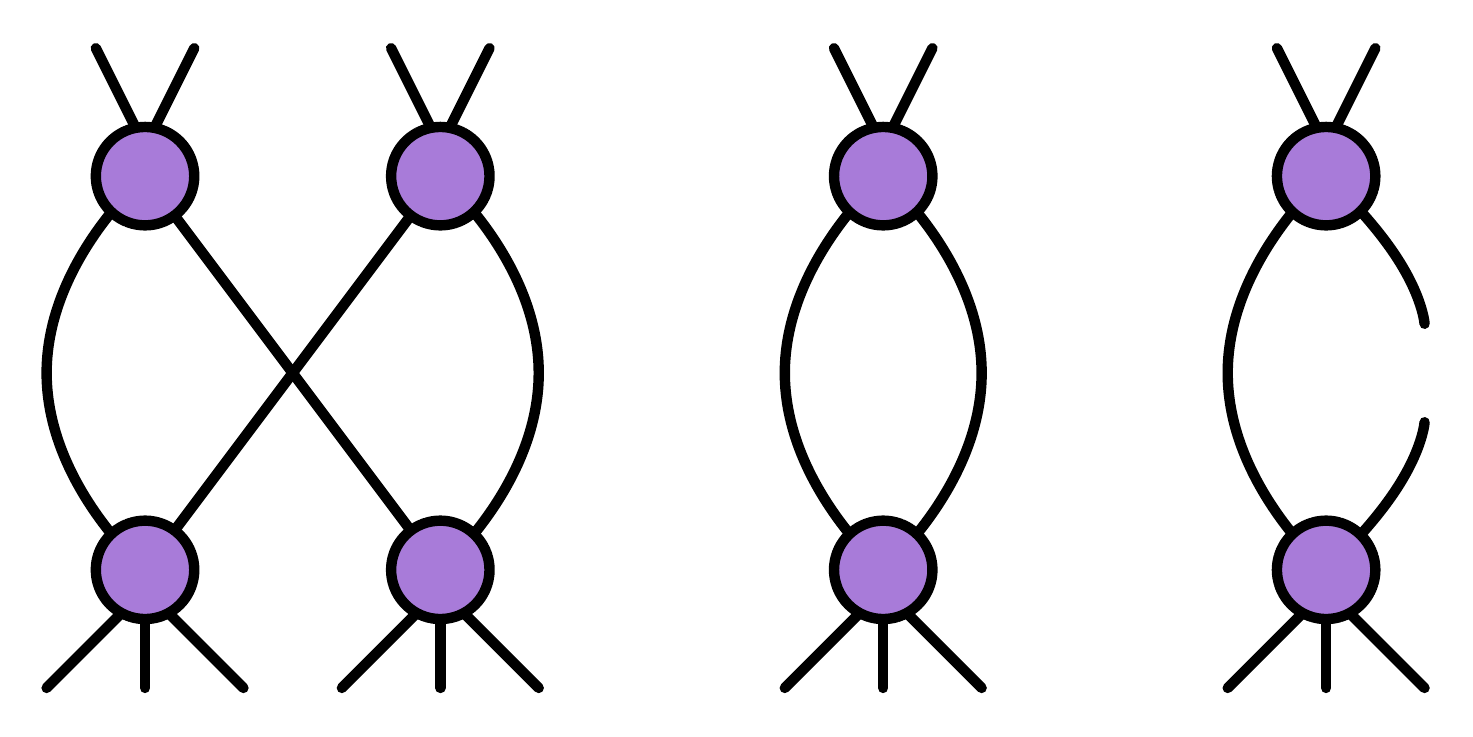}
\caption{Example (directed) \'etale maps $H \rightarrow G \leftarrow K$}
\label{fig dir etale}
\end{figure}
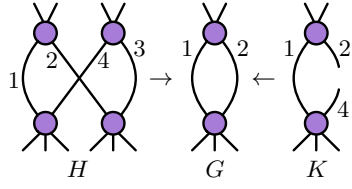

\begin{example}\label{ex dir etale}
In this example, we will refer to two \'etale maps $H  \rightarrow G \leftarrow K$ between directed graphs appearing in \cref{fig dir etale}.
Following our convention, the graphs have vertices with two inputs and two outputs, and also vertices with two inputs and three outputs.
The edges $2$ and $4$ in $H$ do not touch, and this graph is only drawn this way to make the \'etale map easier to see.
Since $G$ has only one $(2,2)$ vertex and one $(2,3)$ vertex, the actions of the maps on vertices are forced.
The actions on the labelled middle edges are given by preservation of parity, and the action on the inputs and outputs of the graphs are the visually expected ones.
We could likewise consider these as \'etale maps of undirected graphs.
These examples show that \'etale maps do not need to be injective on vertices, and even if they are, they do not need to be injective on edges or arcs.
\end{example}

\begin{remark}\label{remark where naive map would be helpful}
\'Etale maps between undirected graphs take paths to paths (one does not even need the pullback condition for this).
It follows that $G \to G'$ is an \'etale map between undirected graphs with $G'$ a tree, then each connected component of $G$ is a tree as well.
Likewise, since \'etale maps between directed graphs preserve directed paths, if $G \to G'$ is such a map with $G'$ acyclic, then each connected component of $G$ is acyclic as well. 
Of course in this case if $G'$ is simply-connected, then so is each component of $G$, by passing to the associated undirected graphs.
\end{remark}

We are mainly interested in connected graphs in what follows, and we make the following definition.
\begin{definition}\label{def embedding}
An \mydef{embedding} between connected graphs is an \'etale map which is injective on vertices.
We denote an embedding by $G \rat G'$.
\end{definition}

Since monomorphisms are stable under pullback, if $G \rat G'$ is an embedding between undirected graphs, then $D_G \to D_G'$ is a monomorphism (and likewise for $I_G \to I_{G'}$ and $O_G \to O_{G'}$ in the directed case).
Embeddings between connected graphs yield injective local homeomorphisms between geometric realizations.
Every subgraph inclusion in the sense of \cref{def subgraph} is an embedding, but embeddings are strictly more general. 
For instance, \cref{fig embedding} indicates two consecutive embeddings (from left to right), where the second indicates the kind of clutching behavior possible in embeddings, which can identify some arcs.
The first is simply a subgraph inclusion, while the second is a bijection on $D$ and $V$, but on arcs is not a monomorphism.

\begin{figure}
\includegraphics[scale=0.3]{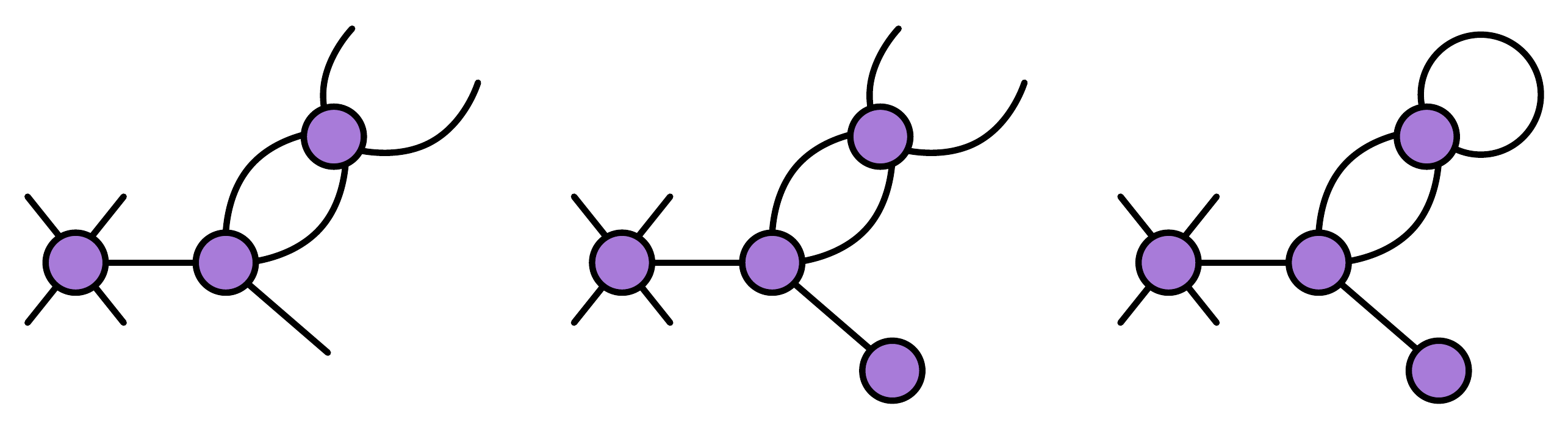}
\caption{Two embeddings}\label{fig embedding}
\end{figure}

Any \'etale map between trees is a monomorphism (see, for instance, \cite[Corollary 4.23]{Raynor:DLCSM}), hence an embedding. 
For acyclic directed graphs the situation is more complicated: there can be \'etale maps between such which are not embeddings, and embeddings which are not injective, as we saw in \cref{ex dir etale} / \cref{fig dir etale}.

\begin{construction}[Lifting embeddings]\label{constr embedding directed}
Suppose $G$ is an undirected graph and $H$ is a directed graph, and temporarily let $U(H)$ be the undirected graph associated to $H$.
If $f\colon G \rat U(H)$ is an embedding, then there is a directed graph $G'$ and an embedding $f' \colon G' \rat H$, so that $G \cong U(G') \rat U(H)$ is the map $f$.
Moreover, the graph $G'$ and embedding $f'$ are essentially unique with respect to this property.
We leave most details to the reader.
Our starting data is the map $f$, as below.
\[ \begin{tikzcd}[sep=small]
A_G \dar & A_G\dar  \lar["\dagger"'] & \lar[hook'] D_G\dar \ar[dr, phantom, "\lrcorner" very near start]  \rar & V_G\dar  
\\ 
E_H \amalg E_H & E_H \amalg E_H \lar["\text{swap}"'] & \lar[hook'] I_H\amalg O_H \rar & V_H
\end{tikzcd} \]
We set $E_{G'} \coloneqq E_G$ and $V_{G'} \coloneqq V_G$, and define $I_{G'}$ and $O_{G'}$ as subsets of $D_G$ fitting into the following pullbacks.
\[ \begin{tikzcd}[sep=small]
I_{G'} \ar[dr, phantom, "\lrcorner" very near start] \dar \rar[hook] & D_G\dar \ar[dr, phantom, "\lrcorner" very near start]  \rar & V_G\dar 
& D_G \lar \dar \ar[dl, phantom, "\llcorner" very near start] & O_{G'} \dar \lar[hook'] \ar[dl, phantom, "\llcorner" very near start]
\\ 
I_H \rar[hook] & I_H\amalg O_H \rar & V_H & \lar I_H\amalg O_H & \lar[hook'] O_H
\end{tikzcd} \]
Using that $f$ is an embedding, one can show that $G'$ is a directed graph (that is, $I_{G'} \to E_{G'}$ and $O_{G'} \to E_{G'}$ are injective).
Further, $f' \colon G' \rat H$ is an embedding. 
\end{construction}

We now introduce a set associated to graphs which will be used to describe all manner of graphical maps below.

\begin{definition}
If $G$ is an directed or undirected connected graph, we define $\emb(G)$ to be the set of embeddings with codomain $G$, modulo isomorphisms in the domain.
\end{definition}
There is a partial order on $\emb(G)$ determined by factorization of embeddings. 
As a partially-ordered set, $\emb(G)$ has a unique maximal element given by the identity of $G$, and each edge determines a minimal element (if $G$ has at least one edge, that is, if it is not $\medstar_0$, then the set of minimal elements is precisely $E_G$).

In light of \cref{constr embedding directed}, if $G$ is a directed graph and $U(G)$ is its associated undirected graph, then $\emb(G) \to \emb(U(G))$ is an isomorphism.
For this reason we do not distinguish notationally between the two: $\emb(G)$ is fundamentally a set associated to an undirected graph.

Every connected subgraph (in the sense of \cref{def subgraph}) of $G$ determines an element of $\emb(G)$, and if $G$ is a tree, then $\emb(G)$ coincides with the set of connected subgraphs.
Thus if $T$ is a subtree of a tree $G$, we will usually write $T\in \emb(G)$ for the class of the associated inclusion $[T \rat G]$.

For an arbitrary graph $G$, a vertex $v$ determines a subgraph $\medstar_v$ just when there are no loops at the vertex $v$, but a vertex always determines an embedding $\iota_v \colon \medstar_v \rat G$ picking out the vertex ($\medstar_v$ is a star with $D_{\medstar_v} = \nbhd(v) \subseteq D_G$).
There are canonical inclusions
\begin{align*}
E_G &\hookrightarrow \emb(G) &
V_G &\hookrightarrow \emb(G)
\end{align*}
of edges and vertices into the embedding set, which we regard as subset inclusions.

\begin{definition}\label{def a union and vertex disjoint}
Suppose that $h\colon H \rat G$ and $k\colon K \rat G$ are two embeddings.
An embedding $\ell \colon L \rat G$ is called \mydef{a union of $h$ and $k$} if 
\begin{enumerate}
\item $[\ell]$ is an upper bound for both $[h]$ and $[k]$ in the poset $\emb(G)$ (that is, there is a factorization $H \rat L \rat G$ of $h$ and likewise for $k$), and \label{def a union ub}
\item $\ell (V_L) = h(V_H) \cup k(V_K)$. \label{def a union cup}
\end{enumerate}
The two embeddings are \mydef{vertex disjoint} if 
\begin{enumerate}[resume*]
\item $h(V_H) \cap k(V_K) = \varnothing$.
\end{enumerate}
\end{definition}

\begin{figure}
\includegraphics[scale=0.3]{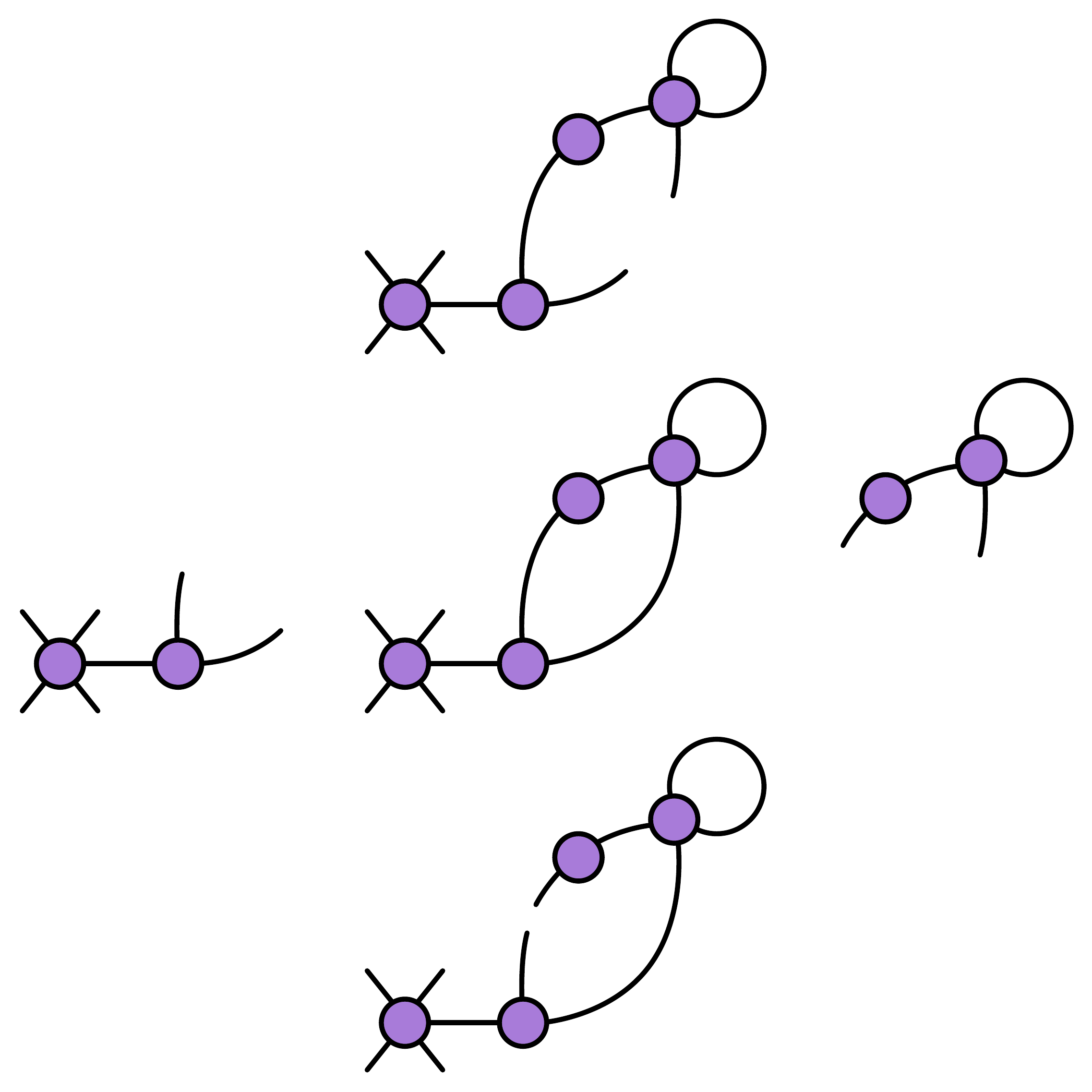}
\caption{Example of non-uniqueness of unions}\label{fig union}
\end{figure}

Unions need not be unique.
In \cref{fig union}, the two subgraphs to the left and right of the graph $G$ in the center have three possible unions in $\emb(G)$.
These are pictured in the center column, and include the maximum $[\id_G] \in \emb(G)$.
Note that these subgraphs do not have a \emph{least} upper bound in $\emb(G)$; also note that least upper bounds that do exist may fail to be unions.
As another example, the embedding pictured in \cref{fig self-union} has \emph{four} possible self-unions, determined by which of the two snipped edges one decides to glue together.

\begin{figure}
\includegraphics[scale=0.3]{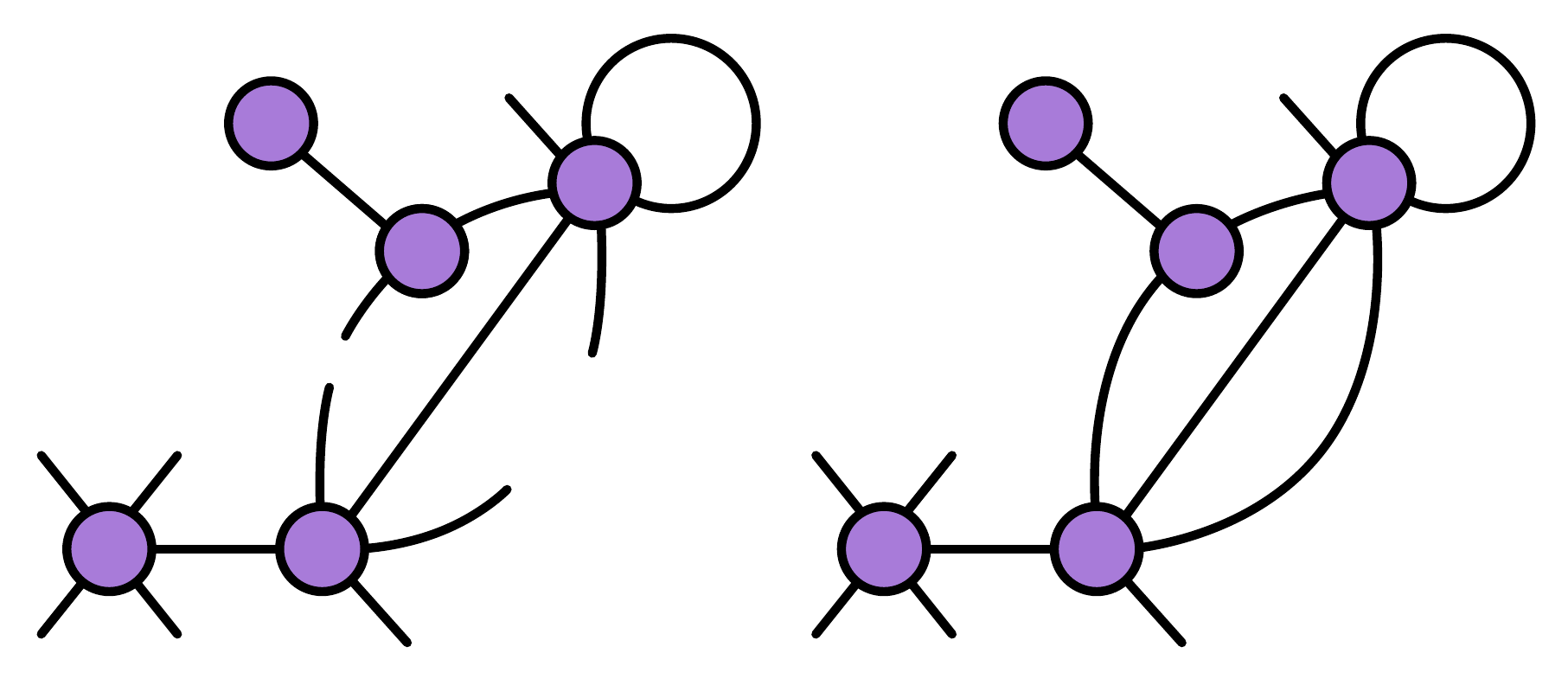}
\caption{Even self-unions are not unique}\label{fig self-union}
\end{figure}

If $H$ and $K$ are subgraphs of a graph $G$ in the sense of \cref{def subgraph}, then there is a natural notion of union subgraph obtained by setting $E_{H\cup K} = E_H \cup E_K \subseteq E_G$ and $V_{H\cup K} = V_H \cup V_K \subseteq V_G$.
If $H$ and $K$ are both connected subgraphs of a connected graph $G$, then $H \cup K$ will be a union in the above sense just when $H$ and $K$ \emph{overlap}, that is, just when at least one of $E_H \cap E_K$ or $V_H \cap V_K$ is inhabited.
In fact, $H\cup K$ is the only union which is also a subgraph, by essentially the same proof as \cite[Lemma B.7]{Hackney:CGOS}. 
In particular, if $G$ is a tree then we know all embeddings come from subtrees, hence any pair $S, T \in \emb(G)$ has at most one union.

\subsection{Boundary-compatibility}
We end this section with one further construction that we will use in our definition of graph maps, that of \emph{boundary-compatibility} in \cref{def boundary compatible}.
If $f \colon H \rat G$ is an embedding between undirected graphs, then the composite
\[
	\eth(H) \hookrightarrow A_H \to A_G 
\]
is a monomorphism (see Lemma 1.20 of \cite{HRY-mod1}), and we write $\eth(f) \subseteq A_G$ for its image, which we call the \mydef{boundary of $f$}.
For example, if $G$ is the graph from \cref{example undirected graph} (see \cref{fig graph example}) and $f$ is the star-embedding picking out the vertex $w$, then $\eth(f) = \{4, 6^\dagger, 5, 5^\dagger\}$.
Likewise, if $f\colon H \rat G$ is an embedding between directed graphs, then $\inp(H) \to E_G$ and $\out(H) \to E_G$ are also monomorphisms,
and we write $\inp(f), \out(f) \subseteq E_G$ for the images of these sets.

The assignment $\eth$ on embedddings descends to a function $\emb(G) \to \wp(A_G)$ landing in the power set of $A_G$.
It is more convenient to think about a multi-set structure, so we instead consider the function
\[
	\eth \colon \emb(G) \to \wp(A_G) \to \mathbb{N}A_G
\]
landing in the free commutative monoid on $A_G$ (whose elements are finite unordered lists of elements of $A_G$).
Likewise, we have functions
\[
	\inp, \out \colon \emb(G) \to \wp(E_G) \to \mathbb{N}E_G
\]
when $G$ is directed.

\begin{definition}\label{def boundary compatible}
Suppose $G$ and $G'$ are two undirected connected graphs, $\varphi_0 \colon A_G \to A_G'$ is an involutive function, and $\hphi \colon \emb(G) \nrightarrow \emb(G')$ is a partially-defined function.
We say that the pair $\varphi \coloneqq (\varphi_0, \hphi)$ is \mydef{boundary-compatible} if the diagram
\[ \begin{tikzcd}
\emb(G) \rar{\eth} \dar[negated]{\hphi}& \mathbb{N}A_G \dar{\mathbb{N} \varphi_0 }\\
\emb(G') \rar{\eth} & \mathbb{N}A_{G'}
\end{tikzcd} \]
commutes on the domain of definition of $\hphi$.
Likewise, if $G$ and $G'$ are directed connected graphs, $\varphi_0 \colon E_G \to E_G'$ is a function, and $\hphi \colon \emb(G) \nrightarrow \emb(G')$ is a partially-defined function, then we say that the pair $\varphi \coloneqq (\varphi_0, \hphi)$ is \mydef{boundary-compatible} if the diagram
\[ \begin{tikzcd}
\mathbb{N}E_G  \dar{\mathbb{N} \varphi_0 } & \emb(G) \rar{\out}\lar[swap]{\inp} \dar[negated]{\hphi}& \mathbb{N}E_G \dar{\mathbb{N} \varphi_0 }\\
 \mathbb{N}E_{G'} & \emb(G') \rar{\out}\lar[swap]{\inp} & \mathbb{N}E_{G'}
\end{tikzcd} \]
commutes on the domain of definition of $\hphi$.
\end{definition}

We are writing the diagrams above using partial maps for conciseness, since we will often want to discuss functions between subsets of $\emb(G)$ and $\emb(G')$ (\cref{def tree map}, \cref{def properadic gcat}, \cref{prop tree maps equiv}), and work with the restrictions of $\eth$ (or $\inp$ and $\out$) to those subsets.

The standard example of a boundary compatible pair is one associated to an embedding $f \colon G\rat G'$.
The function $\varphi_0$ is defined to be $A_G \to A_{G'}$ (or $E_G \to E_{G'}$ if $G$, $G'$ are directed), and the (total) function $\hphi \colon \emb(G) \to \emb(G')$ is defined by post-composition with $f$.

\section{Categories of simply-connected graphs}\label{section trees}

\begin{definition}\label{def tree map}
Suppose $G$ and $G'$ are undirected trees.
\begin{enumerate}
\item A \mydef{tree map} is a boundary-compatible pair $(\varphi_0, \varphi_1)$ consisting of an involutive function $\varphi_0 \colon A_G \to A_{G'}$ and a function $\varphi_1 \colon V_G \to \emb(G')$.
\item Suppose $\varphi = (\varphi_0,\hphi)$ is a pair consisting of an involutive function $\varphi_0 \colon A_G \to A_{G'}$ and a function $\hphi \colon \emb(G) \to \emb(G')$.
We say that $\varphi$ is a \mydef{full tree map} if is boundary-compatible, and, if two subtrees $S$ and $T$ of $G$ overlap, then so do $\hphi(S)$ and $\hphi(T)$ and we have $\hphi(S\cup T) = \hphi(S) \cup \hphi(T)$.
\end{enumerate}
The full tree maps form a category via composition of pairs, which we denote by $\gcatnought$.
Trees with inhabited boundary span the full subcategory $\gcatcyc \subset \gcatnought$.
\end{definition}

Every full tree map determines a tree map by restriction of $\hphi$ along $V_G \hookrightarrow \emb(G)$, and we show in \cref{prop tree maps equiv} of \cref{sec trees vs full trees} that this gives a bijection between tree maps and full tree maps. 
In light of this, we will use the shorter term `tree map' also for full tree maps.
Since subtrees are uniquely determined by their boundaries, boundary-compatibility implies that tree maps send edges to edges.

It turns out that tree maps also preserve \emph{intersections} between overlapping subtrees.
This follows from Theorem A.7 and Proposition B.21 of \cite{Hackney:CGOS}, but we provide an elementary, direct proof in \cref{subsec full tree intersections}. 

We now give an alternative description of $\gcatcyc$ as a full subcategory of the category of \emph{cyclic operads}, $\cyc$, and likewise of $\gcatnought$ as a full subcategory of \emph{augmented cyclic operads}, $\augcyc$.
The distinction between the two is that augmented cyclic operads $P$ are allowed to have a meaningful set $P(\,\,)$ of operations, which cannot compose with any other operations.
The term `augmented cyclic operad' follows \cite{HinichVaintrob:COACD}, and would be called `entries-only $\ast$-polycategories' in \cite{Shulman:2Chu} or simply `cyclic operads' in \cite[Definition 2.3]{DrummondColeHackney:DKHTCO}; see \cite[Example 7.7]{Shulman:2Chu} for some discussion on terminology. 
If $\mathfrak{C}$ is an involutive set of colors, we write $\cyc_{\mathfrak{C}} \subset \augcyc_{\mathfrak{C}}$ for categories of $\mathfrak{C}$-colored objects and identity-on-colors morphisms.

Each tree $G$ determines a free object $\freecyc(G)$ in $\augcyc_{A_G}$.
The generating set of $\freecyc(G)$ is precisely the set $V_G$, where $v\in V_G$ is considered as an element of $\freecyc(G)(\eth(\medstar_v)) = \freecyc(G)(\nbhd(v)^\dagger)$.
A variation of \cite[\S2.2]{HRY-mod2} allows one to show that this becomes a functor $\freecyc \colon \gcatnought \to \augcyc$.

\begin{proposition}\label{prop gcatnought ff}
The category $\gcatnought$ is equivalent to the full subcategory of augmented cyclic operads $\augcyc$ on the objects $\freecyc(G)$ (where $G$ ranges over all trees).
This equivalence restricts to an equivalence between $\gcatcyc$ and the full subcategory of cyclic operads $\cyc \subset \augcyc$ on the objects $\freecyc(G)$, where $G$ ranges over all trees with inhabited boundary.
\end{proposition}
The second part of this is simply the observation that $\freecyc(G)$ is a cyclic operad if and only if $G$ has inhabited boundary.
The first part of the statement will be proved in \cref{subsec tree maps and cyclic maps}.

Turning now to directed trees, we introduce the `dioperadic graphical category.'

\begin{definition}\label{def dir tree map}
Suppose $G$ and $G'$ are directed trees.
A \mydef{tree map} from $G$ to $G'$ is a boundary-compatible pair 
$\varphi = (\varphi_0,\hphi)$ consisting of a function $\varphi_0 \colon E_G \to E_{G'}$ and a function $\hphi \colon \emb(G) \to \emb(G')$ which preserves unions: if $S$ and $T$ overlap, then $\hphi(S\cup T) = \hphi(S) \cup \hphi(T)$.
These form a category of all directed trees, which we denote by $\pgcatsc$.\footnote{
This was called $\pgcat_{\mathrm{sc}}$ in \cite{ChuHackney}, and is equivalent to the category $\mathbf{\Theta}$ in \cite[\S6.3.5]{HRYbook}.}
\end{definition}

Of course one can equivalently define tree maps by replacing $\hphi$ with a function $\varphi_1 \colon V_G \to \emb(G')$, as in \cref{def tree map}.
A version of \cref{prop gcatnought ff} then holds in this setting, with essentially the same proof: the category $\pgcatsc$ is equivalent to a full subcategory (spanned by the directed trees) on the category of dioperads.

By restricting the objects further, we conclude that the dendroidal category from \cite{MoerdijkWeiss:DS}, originally defined as a category of the operads associated to rooted trees, is equivalent to a full subcategory of $\pgcatsc$.

\begin{definition}\label{def dendroidal category}
The Moerdijk--Weiss \mydef{dendroidal category}, denoted by $\dendcat$, is the full subcategory of $\pgcatsc$ consisting of those trees $G$ so that $\out(v)$ is a singleton set for every $v\in V_G$.
The \mydef{simplicial category} $\simpcat$ can be identified with the full subcategory of $\pgcatsc$ on the linear graphs from \cref{ex linear graphs}.
\end{definition}

Of course if one is only interested in rooted trees, simpler descriptions are possible; see, for example, \cite{Kock:PFT}.

\section{On several categories of graphs}

We now extend \cref{def tree map} and \cref{def dir tree map} to maps between graphs that are not simply-connected.
This formulation first appeared in \cite{Hackney:CGOS}.

\begin{definition}\label{def graph map}
Let $G$ and $G'$ be connected undirected (resp.\ directed) graphs.
A \mydef{graph map} $\varphi \colon G \to G'$ is a pair $(\varphi_0, \hphi)$ consisting of a function $\hphi \colon \emb(G) \to \emb(G')$ as well as 
\begin{enumerate}
\item an involutive function $\varphi_0 \colon A_G \to A_{G'}$ if $G$ and $G'$ are undirected, or
\item a function $\varphi_0 \colon E_G \to E_{G'}$ if $G$ and $G'$ are directed.
\end{enumerate}
The pair $(\varphi, \hphi)$ must satisfy the following four conditions:
\begin{enumerate}[label=(\roman*),ref=\roman*]
\item The function $\hphi$ sends edges to edges. \label{graph def edges}
\item The function $\hphi$ preserves unions. \label{graph def union}
\item The function $\hphi$ takes vertex disjoint pairs to vertex disjoint pairs. \label{graph def intersect}
\item The pair $(\varphi, \hphi)$ is boundary-compatible (in the sense of \cref{def boundary compatible}). \label{graph def boundary}
\end{enumerate}
Composition of graph maps is given by composition of the two constituent functions.
This yields a category $\gcat$ of undirected connected graphs and their graph maps, as well as a category $\ocat$ of directed connected graphs and their graph maps.
\end{definition}

\begin{proposition}\label{prop equiv to other graph cats}
The category $\gcat$ of undirected connected graphs and their graph maps is equivalent to the graphical category from \cite[Definition 1.31]{HRY-mod1}.
The category $\ocat$ of directed connected graphs and their graph maps is equivalent to the wheeled properadic graphical category from \cite[\S2]{HRYfactorizations}.
\end{proposition}
\begin{proof}
This result appears in \cite{Hackney:CGOS}: the first part is Theorem 4.15, and the second part combines Proposition 5.10 and Theorem 5.13.
`Wheeled properadic graphical category' refers to the category $\mathcal{B}$ from \cite{HRYfactorizations}.
\end{proof}

The next proposition uses that tree maps send edges to edges and preserve vertex-disjoint pairs (\cref{intersection lemma}).

\begin{proposition}\label{prop tree full subcats}
The category $\gcatnought$ is the full subcategory of $\gcat$ on the undirected trees.
The category $\pgcatsc$ is the full subcategory of $\ocat$ on the directed trees. \qed
\end{proposition}

\begin{remark}\label{rmk directed graphs as slice}
If $\CC$ is a category, let $\widehat{\CC} \coloneqq \fun(\CC^\oprm, \set)$ be the category of set-valued presheaves.
If $X \in \widehat{\CC}$ is a presheaf and $G \in \CC$ is an object, we will write $X_G$ for the value of $X$ at $G$.
The Yoneda embedding $\CC \to \widehat{\CC}$ takes an object $G$ to the representable presheaf $\CC(-,G)$.
An important example, when $\CC = \gcat$ or some full subcategory, is the \mydef{orientation presheaf} $\opshf$.
This is defined by 
\[
	\opshf_G = \{ \text{involutive functions } A_G \to  \{+1, -1\} \}
\]
where the set $\{+1,-1\}$ is equipped with the free involution.
This additional data determines a directed structure on $G$, see \cite[Construction 5.4]{Hackney:CGOS}.
Meanwhile, every directed graph determines a canonical element of the orientation presheaf.
In fact, we have equivalences of categories 
\[ \begin{tikzcd}[row sep=0]
\ocat \rar{\simeq} & \gcat \times_{\widehat{\gcat}} \widehat{\gcat}_{/\opshf} \\
\pgcatsc \rar{\simeq} & \gcatnought \times_{\widehat{\gcatnought}} \widehat{\gcatnought}_{/\opshf}
\end{tikzcd} \]
whose first projections forget the directed structure.
In other words, the category of elements of $\opshf \in \widehat{\gcat}$ (resp.\ $\opshf \in \widehat{\gcatnought}$) is $\ocat$ (resp.\ $\pgcatsc$).
See \cite[Proposition 5.10]{Hackney:CGOS}.
\end{remark}
A similar construction appears in the context of Feynman categories and Borisov--Manin graph morphisms in \cite[6.4.1]{KaufmannLucas:DFC}, and for \'etale maps in \cite[\S4.5]{Raynor:DLCSM}.

\begin{definition}[Active and inert maps]\label{def active inert maps}
A graph map $\varphi = (\varphi_0, \hphi) \colon G \to G'$ is an \mydef{active map} if $\hphi \colon \emb(G) \to \emb(G')$ preserves the maximum ($\hphi([\id_G]) = [\id_{G'}]$).
We write $\gcatact \subset \gcat$ and $\ocat_{\actrm} \subset \ocat$ for the wide subcategories of active maps, and decorate the arrows as $G \ract G'$.
A graph map $\varphi$ is \mydef{inert} if it comes from an embedding; in other words, if the the dashed map in the diagram
\[ \begin{tikzcd}
V_G \dar[hook] \rar[dashed] & V_{G'} \dar[hook] \\
\emb(G) \rar{\hphi} & \emb(G')
\end{tikzcd} \]
exists. 
We write $\gcatemb \subset \gcat$ and $\ocat_\intrm \subset \ocat$ for the subcategories of inert maps, and decorate the arrows by $G \rat G'$.
\end{definition}

Every graph map $\varphi \colon G \to G'$ factors as an active map followed by an inert map.
Indeed, if $\hphi([\id_G]) = [H \rat G']$, then the factorization takes on the form $G \ract H \rat G'$.
Constructing the active map $G\ract H$ is delicate in general since $A_H \to A_{G'}$ (resp.\ $E_H \to E_{G'}$) need not be a monomorphism.
Nevertheless, we have the following (see \cite[Theorem 2.15]{HRY-mod1}).
\begin{theorem}\label{thm fact system}
The pair $(\gcatact, \gcatemb)$ is an orthogonal factorization system\footnote{An \mydef{orthogonal factorization system} on a category $\CC$ is a pair $(L, R)$ of classes of maps, each closed under composition and containing all of the isomorphisms. These have the property that any morphism of $\CC$ factors as $r \circ \ell$, with $r\in R$ and $\ell\in L$, and this factorization is unique, up to unique isomorphism.} on the category $\gcat$, as is $(\ocat_{\actrm}, \ocat_\intrm)$ on $\ocat$.
These factorization systems restrict to factorization systems on the subcategories $\gcatcyc \subset \gcatnought \subset \gcat$ and $\simpcat \subset \dendcat \subset \pgcatsc \subset \ocat$. \qed
\end{theorem}

\subsection{Properadic graphical category}\label{subsec properadic graph cat}
We briefly describe the properadic graphical category $\pgcat$, whose objects are connected \emph{acyclic} directed graphs, in the sense of \cref{def directed paths acyclic}.
It requires a somewhat subtle notion of \emph{structured subgraph}, which we give in a reformulation first appearing as \cite[1.6.5]{Kock:GHP}; see also \cite[Remark 2.2.4]{ChuHackney}.

\begin{definition}[Structured subgraph]\label{def ssub}
A \mydef{na\"ive morphism of directed graphs} has the same data as an \'etale map \cref{def etale}, but we do not require the middle squares to be pullbacks.
Suppose $G$ is an acyclic directed graph.
A connected subgraph $H\subseteq G$ is a \mydef{structured subgraph} if the inclusion $H\rat G$ is right orthogonal in the category of na\"ive morphisms of directed graphs to the maps ${\downarrow} \amalg {\downarrow} \to L_n$ (for all $n\geq 0$) picking out the input and output edge in the linear graph $L_n$.
In other words, given a commutative square as below (whose horizontal arrows may just be na\"ive morphisms), there is a unique dashed lift.
\[ \begin{tikzcd}
{\downarrow} \amalg {\downarrow} \rar \dar[swap]{0,n}  & H \dar[tail] \\
L_n \rar \ar[ur,dashed,"\exists !"] & G
\end{tikzcd} \]
We write $\ssub(G) \subseteq \emb(G)$ for the set of structured subgraphs of $G$.
\end{definition}

This can alternately be phrased in terms of `graph substitution' (see \cite[Remark 2.2.4]{ChuHackney}).
Every edge and star subgraph is a structured subgraph, so we have $E_G, V_G \subseteq \ssub(G)$.
Note that if $T$ is a \emph{simply-connected} directed graph (i.e., an object in $\pgcatsc$), then every connected subgraph is a structured subgraph, so we may identify $\ssub(T)$ with $\emb(T)$.

The following formulation is due to Hongyi Chu and the author \cite[Definition 2.2.11]{ChuHackney}, who also showed it is equivalent to the graphical category from \cite[Chapter 6]{HRYbook}.
This description is akin to the simpler \cref{def dir tree map}, but predates it.

\begin{definition}\label{def properadic gcat}
The \mydef{properadic graphical category}, denoted $\pgcat$, has objects the acyclic directed graphs.
A morphism $\varphi \colon G \to G'$ consists of a pair of functions $\varphi_0 \colon E_G \to E_{G'}$ and $\hphi \colon \ssub(G) \to \ssub(G')$ so that $(\varphi_0, \hphi)$ is boundary-compatible and $\hphi$ preserves unions.
By `preserving unions,' we mean that if $H_1, H_2 \in \ssub(G)$ are such that $H_1 \cup H_2$ is a structured subgraph, then $\hphi(H_1 \cup H_2) = \hphi(H_1) \cup \hphi(H_2)$.
Composition in $\pgcat$ is given by composition of pairs.
\end{definition}

It turns out that given a map in $\pgcat$, one can extend $\hphi \colon \ssub(G) \to \ssub(G')$ to a function $\emb(G) \to \emb(G')$ in a canonical way.
This is one ingredient in the following (see \cite{Hackney:CGOS} for a proof).

\begin{theorem}\label{thm properadic gcat}
The properadic graphical category $\pgcat$ may be identified with the subcategory of $\ocat$ with
\begin{enumerate}
\item objects the acyclic directed graphs, and \label{properadic gcat objs}
\item morphisms those $\varphi \colon G \to G'$  so that $\hphi (G) \in \ssub(G')$. \label{properadic gcat mors}
\end{enumerate}
\end{theorem}

Since $\ssub(G)$ and $\emb(G)$ coincide when $G$ is simply-connected, \cref{prop tree full subcats} implies the following (or directly compare \cref{def dir tree map} and \cref{def properadic gcat}).

\begin{corollary}
The full subcategory of $\pgcat$ on the directed trees is $\pgcatsc$.  \qed
\end{corollary}

The notions of active and inert maps from \cref{def active inert maps} can be modified in this context, where now maps in $\pgcat_{\intrm}$ are those maps which are isomorphic to a structured subgraph inclusion.
This again is a factorization system $(\pgcat_\actrm, \pgcat_\intrm)$, restricted from the one on $\ocat$ from \cref{thm fact system} (see \cite[2.4.14]{Kock:GHP}).

\section{The Segal condition and generalized operads}
\label{sec segal}

In this section, we turn to the main topic of the paper, which is the interpretation of generalized operads as Segal presheaves for appropriate graph categories.

First, let us give an idea about how to understand certain elements in a presheaf. 
Let $X \in \widehat{\gcat}$ be an $\gcat$-presheaf.
We interpret the set $X_{\updownarrow}$ as a set of colors.
Since $\updownarrow$ has a non-trivial automorphism in $\gcat$, the set $X_{\updownarrow}$ attains a potentially non-trivial involution. (If instead $X$ was an $\ocat$-presheaf, then we would not expect to have an involution on $X_{\downarrow}$, as ${\downarrow}\in \ocat$ only possesses the identity automorphism.)
We interpret $X_{\medstar_n}$ as operations with arity $n$.
Using the maps ${\updownarrow} \to \medstar_n$ classifying each boundary arc in the $n$-element set $\eth(\medstar_n) = \mathbf{n}$, we obtain a function
\[
	X_{\medstar_n} \to \prod_{\mathbf{n}} X_{\updownarrow}
\]
giving the boundary profile of an operation.
(Likewise, if $X\in \widehat{\ocat}$ then we have functions $X_{\medstar_{n,m}} \to \prod_{\mathbf{n}} X_{\downarrow}$ and $X_{\medstar_{n,m}} \to \prod_{\mathbf{m}} X_{\downarrow}$ giving the inputs and outputs.)

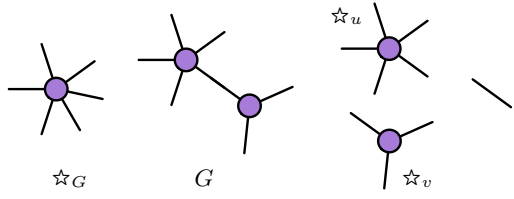
\begin{figure}
\labellist
\small\hair 2pt
\pinlabel $\medstar_u$ [l] at 410 230
\pinlabel $G$ [l] at 240 26
\pinlabel $\medstar_G$ [l] at 60 26
\pinlabel $\medstar_v$ [l] at 500 26
\endlabellist
\centering
\includegraphics[scale=0.3]{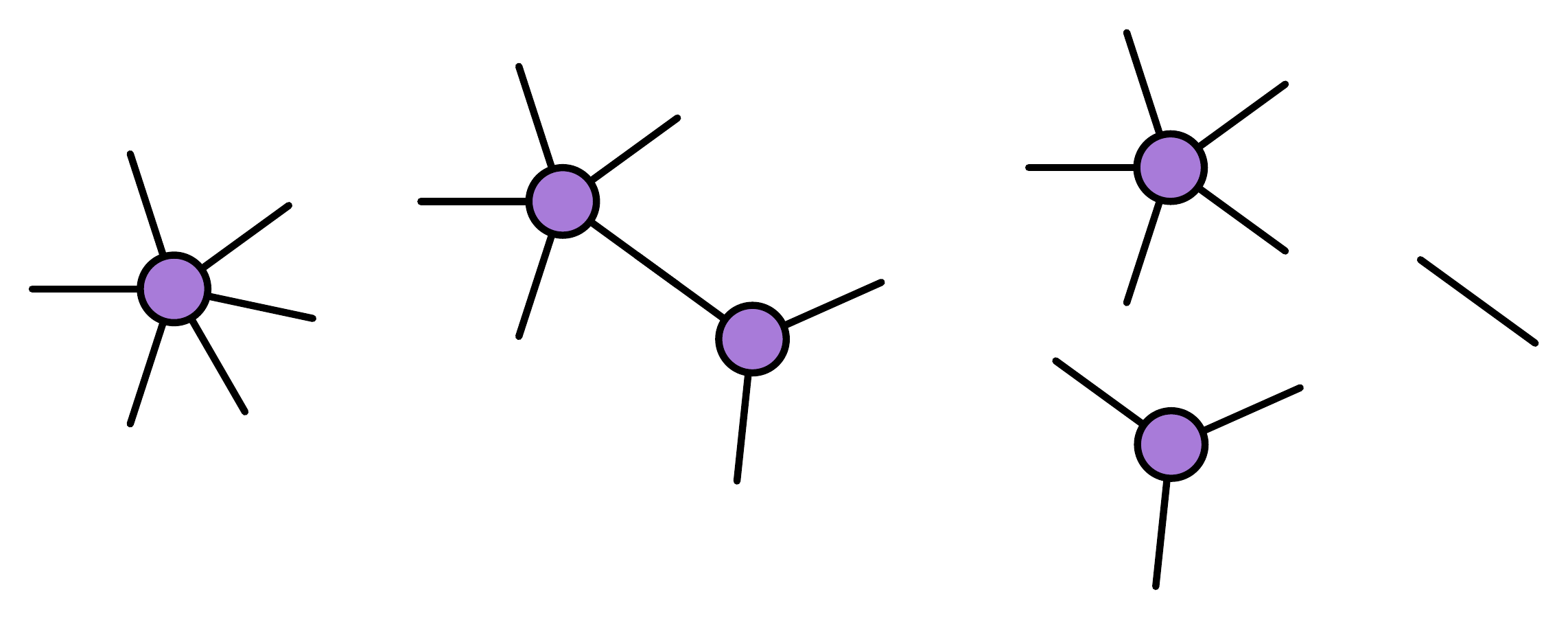}
\caption{Graph $G$ with two vertices and one internal edge}
\label{fig two vertex simple}
\end{figure}

Suppose now that $G \in \gcat$ has two vertices $u,v$ and exactly one internal edge, as in \cref{fig two vertex simple}.
We then have a span of sets as follows:
\begin{equation}\label{eq seg map two vertex simple} \begin{tikzcd}[sep=small]
& X_G \ar[dl, end anchor=north east] \ar[dr, end anchor=north west]  \\
X_{\medstar_G} & & X_{\medstar_u} \underset{X_{\updownarrow}}\times X_{\medstar_v}
\end{tikzcd} \end{equation}
The leftward leg is associated to an active map $\medstar_G \ract G$ from a star graph (we may take $\medstar_G$ to be a star with $\eth(\medstar_G) = \eth(G)$), and the rightward leg is associated to the two inert maps $\iota_u \colon \medstar_u \rat G$ and $\iota_v \colon \medstar_v \rat G$.
This span represents a kind of generalized way of `composing' operations $X$.
If the leg on the right happens to be a bijection, then this is a legitimate function
\[
	X_{\medstar_G} \leftarrow X_{\medstar_u} \times_{X_{\updownarrow}} X_{\medstar_v} 
\]
giving \emph{composition} of two operations in $X$.
This is the kind of fundamental operation one has in a(n augmented) cyclic operad.
In the directed setting, the evident variation gives a dioperadic composition.\footnote{When considering a $\pgcat$-presheaf $X$, it is more appropriate to work with acyclic graphs $G$ having two vertices, but an arbitrary number of internal edges, resulting in a span of the form $X_{\medstar_G} \leftarrow X_G \rightarrow X_{\medstar_u} \times_{\prod X_{\downarrow}} X_{\medstar_v}$, where $\prod X_{\downarrow}$ is the product indexed by the inner edges of $G$.}

Instead, suppose that $G\in \gcat$ has one vertex $v$ and exactly one internal edge, which is necessarily a loop at $v$, as in \cref{fig one vertex simple}.
We can again form a span 
\begin{equation}\label{eq seg map one vertex simple} \begin{tikzcd}[sep=small]
& X_G \ar[dl, end anchor=north east] \ar[dr, end anchor=north west]  \\
X_{\medstar_G} & & \widetilde{X_{\medstar_v}} \rar{\subset} & X_{\medstar_v}
\end{tikzcd} \end{equation}
whose right leg lands in the subset $\widetilde{X_{\medstar_v}}$ of $X_{\medstar_v}$ consisting of those elements which agree along the two embeddings $j_{a}, j_{a'} \colon {\updownarrow} \rat \medstar_v$ with $\iota_v j_a = \iota_v j_{a'} \colon {\updownarrow} \rat \medstar_v \rat G$.
If this rightward map is a bijection, this gives a \emph{contraction} of suitable operations
\[
	X_{\medstar_G} \leftarrow \widetilde{X_{\medstar_v}}.
\]
These are precisely the sort of contractions that one would expect to see in a modular operad.
In the directed case, the loop must go from an output of the vertex to an input, and we arrive at the contractions for a wheeled properad.

\begin{figure}
\labellist
\small\hair 2pt
\pinlabel $G$ [l] at 220 0
\pinlabel $\medstar_G$ [l] at 70 0
\pinlabel $\medstar_v$ [l] at 400 0
\endlabellist
\centering
\includegraphics[scale=0.3]{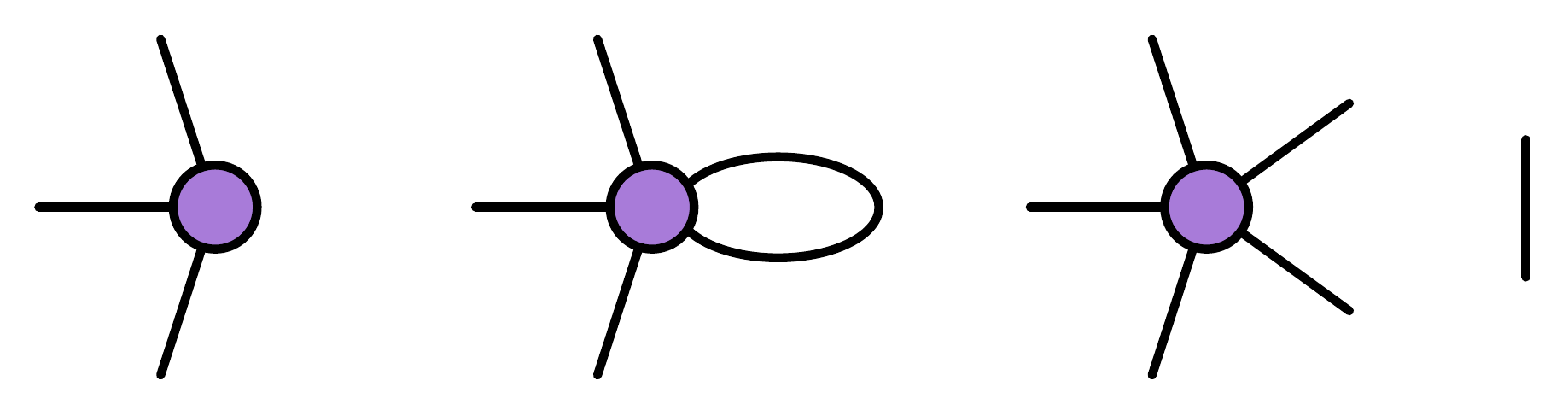}
\caption{Graph with one vertex and one internal edge}
\label{fig one vertex simple}
\end{figure}

There are of course axioms that these compositions and contractions should satisfy (unitality, associativity, etc.), which can likewise be governed by graphs.
Let us give the full abstract definition.

\begin{definition}[Segal presheaves]\label{def segal condition}
Suppose $\CC$ is one of the graph categories we have discussed above ($\gcat, \gcatnought, \gcatcyc, \ocat, \pgcat, \pgcatsc, \dendcat, \simpcat$), and let $\CC_\elrm \subseteq \CC_\intrm$ denote the category whose objects are \emph{elementary graphs}, that is, edges and stars, and whose morphisms are inert maps. 
\begin{itemize}
\item If $G\in \CC$ is a graph, write $\CC_{/G}^\elrm \coloneqq \CC_\elrm \times_{\CC} \CC^\intrm_{/G}$ for the category whose objects are inert maps $K \rat G$ from an elementary object $K$ to $G$, and whose morphisms are inert maps.
\item If $X\in \widehat{\CC}$ is a presheaf, then the \mydef{Segal map} at $G$ is the function
\[
	X_G \to \lim_{\substack{K \rat G  \\ \in(\CC_{/G}^\elrm)^\oprm}} X_K. 
\]
\item A presheaf $X\in \widehat{\CC}$ is said to satisfy the \mydef{Segal condition} if the Segal map is a bijection for all $G\in \CC$.
\end{itemize}
We also say that $X$ is a \mydef{Segal presheaf} in this case.
We write $\seg(\CC) \subset \widehat{\CC}$ for the full subcategory on the Segal presheaves.
\end{definition}

One can unravel that if $\CC = \gcat$, then the rightward maps in the spans \eqref{eq seg map two vertex simple} and \eqref{eq seg map one vertex simple} are equivalent to the Segal maps for these graphs.
This uses a final subcategory argument (we can omit the objects of $\CC_{/G}^\elrm$ corresponding to boundary edges without changing the limit).
Likewise, if $\CC = \simpcat$, the Segal condition reduces to the one at the beginning of the introduction (the Segal maps on $X_1$ and $X_0$ are always identities).
As an exercise, the reader should determine for which graph categories $\CC$ all of the representable presheaves $\CC(-,G)$ are Segal.

\subsection{The general situation} There are now abstract settings for Segal conditions, which cover many situations of interest, including all of those above.
Chu and Haugseng gave an extensive study of the abstract Segal condition for $\infty$-categories in \cite{ChuHaugseng:HCASC}, and the following appears as Definition 2.1 therein.
\begin{definition}\label{def alg pattern}
An \mydef{algebraic pattern} is an $\infty$-category $\PP$ equipped with an inert-active factorization system 
$(\PP_\intrm, \PP_\actrm)$, along with some specified full subcategory $\PP_\elrm \subseteq \PP_\intrm$ whose objects are called \mydef{elementary}.
\end{definition}
Given a functor $X \colon \PP \to \mathcal{S}$, one says that $X$ is Segal if $X(p) \to \lim_{e\in (\PP_\elrm)_{/p}} X(e)$ (where here the limit is in the $\infty$-categorical sense) is an equivalence for all $p\in \PP$.
If $\CC$ is any one of the graph categories under discussion, then $\PP = \CC^\oprm$ is an algebraic pattern in this sense, and this condition agrees with that from \cref{def segal condition}.

Less general are the hypermoment categories of Berger \cite{Berger:MCO}, a framework that still includes the graph categories.
To state the definition, recall that the skeletal category of finite pointed sets $\mathbf{F}_\ast$ has an inert-active factorization system \cite[Remark 2.1.2.2]{Lurie:HA}. 
A map $\{ \ast, 1, 2, \dots, n\} = \lr{n} \to \lr{m}$ is \emph{active} when only the basepoint maps to the basepoint, and \emph{inert} if exactly one element of $\lr{n}$ maps to each non-basepoint element of $\lr{m}$.
We are interested in the corresponding active-inert factorization system on $\fstarop$. 
\begin{definition}[Definition 3.1 of \cite{Berger:MCO}]\label{def hypermoment}
A (unital) \mydef{hypermoment category} consists of a category $\CC$, an active-inert factorization system on $\CC$, and a functor $\gamma\colon \CC \to \fstarop$ respecting the factorization systems.
These data must satisfy the following:
\begin{enumerate}
\item For each $c \in \CC$ and each inert map $\lr{1} \rat \gamma(c)$ in $\fstarop$, there is an essentially unique inert lift $u \rat c$ where $u$ is a unit. Here, \emph{unit} means that $\gamma(u) = \lr{1}$ and any active map with codomain $u$ has precisely one inert section. \label{def hm inert lifts} 
\item For each $c\in \CC$, there is an essentially unique active map $u \ract c$ whose domain is a unit.\label{def hm active}
\end{enumerate}
\end{definition}
Given a hypermoment category $\CC$, one can define $\CC_\elrm \subset \CC_\intrm$ to be the full subcategory spanned by the units and the \emph{nilobjects}, those objects living over $\lr{0} \in \fstarop$.
Imitating \cref{def segal condition} gives a notion of Segal presheaf on $\CC$. 

The graph categories we have discussed so far become hypermoment categories by means of the functor $\vertfunc \colon \CC \to \finsetstarop$ that sends a graph $G$ to $(V_G)_+$.
If $\varphi \colon G \to G'$ is a morphism and $w\in V_{G'}$, then $\vertfunc(\varphi)(w) = v$ just when $w$ is a vertex in $\hphi(\medstar_v)$.
Otherwise $w$ is sent to the basepoint of $(V_G)_+$.
By vertex-disjointness, this is a well-defined function $(V_{G'})_+ \to (V_G)_+$.
The units are precisely the stars, and the maps $\medstar_G \ract G$ are those from \eqref{def hm active}.
See \cite{Berger:MCO,Hackney:CGOS} for details.

\subsection{Nerves}
In \cref{section trees}, we saw a construction taking a tree $G$ to the (augmented) cyclic operad $\freecyc(G)$ which is freely generated by the vertices and arcs of $G$.
This induced a (fully faithful) functor $\gcatnought \to \augcyc$. 
This construction be imitated for each of the graph categories, yielding the following collection of functors into categories of generalized operads.

\begin{center}
\begin{tabular}{@{}llll@{}} \toprule
\multicolumn{2}{c}{Directed} & \multicolumn{2}{c}{Undirected} \\
\cmidrule(r){1-2} \cmidrule(r){3-4}
Functor & Structure & Functor & Structure \\
\midrule
$\ocat \to \wproperad$ & Wheeled properad & $\gcat \to \modular$ & Modular operad 
\\
$\pgcat \to \properad$ & Properad &  &  
\\
$\pgcatsc \hookrightarrow \diop$ & Dioperad & $\gcatnought \hookrightarrow \augcyc$ & Aug.\ cyclic operad
\\
 &  & $\gcatcyc \hookrightarrow \cyc$ & Cyclic operad
\\
$\dendcat \hookrightarrow \operad$ & Operad &  &  
\\
$\simpcat \hookrightarrow \cat$ & Category & & 
\\ \bottomrule
\end{tabular}
\end{center}

\begin{remark}\label{rmk Raynor Kock categories}
Most of the functors in the preceding table are fully faithful, with the exceptions of $\ocat \to \wproperad$, $\gcat \to \modular$, and $\pgcat \to \properad$.
We can uniquely factor each of them into an identity-on-objects functor, followed by a fully faithful functor.
Two of these,
\[ \begin{tikzcd}[row sep=0]
\gcat \rar{\text{iden-on-obj}} &[+0.75cm] \mathbf{R} \rar[hook, "\text{f.f.}"] & \modular \\
\pgcat \rar{\text{iden-on-obj}} &[+0.75cm] \mathbf{K} \rar[hook, "\text{f.f.}"] & \properad
\end{tikzcd} \]
have appeared in the work of Raynor (see \cite[\S8.4]{Raynor:DLCSM}, where $\mathbf{R}$ is called $\Xi$) and Kock (see \cite[2.4.14]{Kock:GHP}, where $\mathbf{K}$ is called $\widetilde{\text{\textbf{\textsl{Gr}}}}$).
These sources include detailed descriptions of working in these categories, which are somewhat more complicated in structure than those they extend.
For example, these categories possess \emph{weak} factorization systems, in contrast with the orthogonal factorization systems present on $\gcat$ and $\pgcat$.
Further, there are not any clear functors into $\fstarop$, since vertex-disjointness fails (see \cite[Remark 7.1.10]{ChuHackney}).
\end{remark}

Each of the functors in the table induces a \mydef{nerve functor}.
For example, suppose $\freemod \colon \gcat \to \modular$ is the functor that takes an undirected graph $G$ to the $A_G$-modular operad $\freemod(G)$ freely generated by its vertices.
Then we can define a functor 
\[ 
N \colon \modular \to \widehat{\gcat} \]
by $NP = \modular(\freemod(-),P) \colon \gcat^\oprm \to \set$.
Practically by definition, we have that
\[
	\colim_{\substack{K \rat G  \\ \in \gcat_{/G}^\elrm}}  \freemod(K) \to \freemod(G)	
\]
is an isomorphism in $\modular$, which formally implies that $NP$ is a Segal presheaf.
This construction and argument applies to all of the graph categories and functors in the table above, giving nerve functors as follows.
\begin{align*}
\modular &\to  \widehat{\gcat} &
    \wproperad &\to \widehat{\ocat} &
    \operad &\to \widehat{\dendcat} \\
\augcyc &\to \widehat{\gcatnought} &
    \properad &\to \widehat{\pgcat}  &
    \cat &\to \widehat{\simpcat} \\
\cyc &\to \widehat{\gcatcyc} &
    \diop &\to \widehat{\pgcatsc} 
\end{align*}
Each of these factors through a full subcategory $\seg(\CC) \hookrightarrow \widehat{\CC}$ of Segal presheaves.
Generalizing the classical equivalence $\cat \simeq \seg(\simpcat)$ from the beginning of the introduction, we have the following.
\begin{theorem}[Nerve theorem]\label{thm nerve}
The nerve functors 
\[
\wproperad \to \seg(\ocat), \properad \to \seg(\pgcat), \operad \to \seg(\dendcat), \text{and }\modular \to \seg(\gcat)
\]
are equivalences of categories.
\end{theorem}
\begin{proof}
These appear in \cite{HRYfactorizations}, \cite{HRYbook}, \cite{Weber:F2FPRA}, and \cite{HRY-mod1}.
\end{proof}

\begin{slogan}
In the presence of the nerve theorem, one can work entirely at the level of Segal presheaves, rather than with the original operadic structure.
\end{slogan}

\begin{remark}\label{rmk diop cyc nerves}
It is expected that the three missing cases $\diop \to \seg(\pgcatsc)$, $\augcyc \to \seg(\gcatnought)$, and $\cyc \to \seg(\gcatcyc)$ are also equivalences. 
I have been told that the latter two equivalences will appear in the PhD thesis of Patrick Elliott, along with a comparison with the `strict inner-Kan condition.' 
Meanwhile, the nerve theorem for dioperads follows from the nerve theorem for augmented cyclic operads (see \cref{prop cyc diopd} below).
Alternatively, in light of \cref{prop gcatnought ff} (and its dioperadic analogue), it is expected that one can bring to bear the tools of \emph{abstract nerve theory} to recover these three cases.
\end{remark}

The aforementioned abstract nerve theory originated in work of Weber \cite{Weber:F2FPRA} and unpublished work of Leinster.
See also \cite{BergerMelliesWeber:MAAT,BourkeGarner:MT}, or the overview of the key points appearing in \cite[2.1]{Raynor:DLCSM}.
In our restricted setting, one can consider $\widehat{\gcat_\elrm}$ as a category of \emph{graphical species} (\cite[\S4]{JoyalKock:FGNTCSM} and \cite[\S1.1]{Raynor:DLCSM}) and $\widehat{\ocat_\elrm}$ as \emph{digraphical species} \cite[\S2.1]{Kock:GHP}, which are a kind of collection with morphisms built in for the coloring maps.
Then $\modular$ can be considered as the Eilenberg--Moore category of algebras $\operatorname{Alg}(T)$ for the free modular operad monad $T$ on $\widehat{\gcat_\elrm}$. 
For particularly nice monads (strongly cartesian), one can automatically produce from this set-up a full subcategory $\mathbf{\Theta}_T \subset \operatorname{Alg}(T)$ which satisfies a nerve theorem $\operatorname{Alg}(T) \xrightarrow{\simeq} \seg(\mathbf{\Theta}_T) \subset \widehat{\mathbf{\Theta_T}}$.
This applies to produce, for instance, $\simpcat$, $\dendcat$ (see Example 2.14 and Example 4.19 of \cite{Weber:F2FPRA}), and also other important categories such as Joyal's cell category $\mathbf{\Theta}_n$.

\begin{remark}\label{rmk KR nerves}
There are also nerve theorems for the categories $\mathbf{K}$ and $\mathbf{R}$ from \cref{rmk Raynor Kock categories}.
The nerve functors (with a similar definition as above)
\[
	N \colon \modular \to \widehat{\mathbf{R}} \qquad N\colon \properad \to \widehat{\mathbf{K}}
\]
are fully faithful, and the essential image of each is the subcategory of Segal presheaves.
The first of these is \cite[Theorem 8.2]{Raynor:DLCSM} (see also \cite[\S4]{HRY-mod2}), while the second is \cite[Theorem 2.3.9]{Kock:GHP}.
Both of these results utilize the framework for abstract nerve theorems, 
though there are intricate subtleties in both instances, and neither is a straightforward application of existing theory.
\end{remark}

\begin{question}\label{question disconnected}
So far in the main text of this paper, we have not mentioned props at all \cite{MacLane:CA,HackneyRobertson:OCP}, yet these are one of the most widespread and expressive kinds of operadic structures.
Nor have we mentioned wheeled props \cite{MarklMerkulovShadrin} or nc (non-connected) modular operads \cite[2.3.2]{KaufmannWard:FC}. 
All of these structures are based on disconnected graphs.
A major motivation for the development of \cref{def graph map} was to facilitate extensions of the categories $\pgcat$, $\ocat$, and $\gcat$ to categories of disconnected graphs $\pgcat_\mathrm{dis}$, $\ocat_\mathrm{dis}$, and $\gcat_\mathrm{dis}$.
Specifically, we hope for an appropriate notion of \emph{embedding} between disconnected graphs, so that we can simply imitate \cref{def graph map} directly to obtain $\ocat_\mathrm{dis}$ and $\gcat_\mathrm{dis}$ (and similarly a disconnected version of structured subgraph, to obtain $\pgcat_\mathrm{dis}$).
We give a wish-list for such a construction:
\begin{enumerate}
\item Each of the three categories would have a naturally occurring active-inert orthogonal factorization system, with inert maps the disconnected embeddings (or the appropriate maps for $\pgcat_\mathrm{dis}$).
\item For each of these categories, there is a hypermoment category structure as in \cref{def hypermoment}, given by a vertex functor $\vertfunc$ into $\finsetstarop$.
In particular, there is a natural notion of Segal presheaf.
\item Each of these categories would be a dualizable generalized Reedy category in the sense of \cite{bm_reedy}, as is the case for their connected counterparts (\cite[\S6.4]{HRYbook}, \cite[Theorem 1.2]{HRYfactorizations}, \cite[\S2.2]{HRY-mod1})
\item The nerve theorem holds, that is, we may regard the categories of props, wheeled props, and nc modular operads as the categories of Segal presheaves $\seg(\pgcat_\mathrm{dis})$, $\seg(\ocat_\mathrm{dis})$, and $\seg(\gcat_\mathrm{dis})$.
\end{enumerate}
Such categories would provide a sound foundation for an approach to homotopy-coherent versions of each operadic structure mentioned, either by appropriately localizing the Reedy model structure on simplicial presheaves, or by localizing $\infty$-category of Segal presheaves $\seg^\infty(\CC) \subset \fun(\CC^\oprm, \mathcal{S})$ of \cite{ChuHaugseng:HCASC} to impose a Rezk-completeness condition (as in \cite{ChuHaugseng:EIO,ChuHackney}).
We note that an entirely different approach to $\infty$-props has recently been proposed in \cite{HaugsengKock}.
\end{question}

These proposed categories of disconnected graphs should also be related to the double category of graphs of \cite[2.5]{BergerKaufmann:TGA}, which has as a key component the total dissection.
There is also the category of forests from \cite{HeutsHinichMoerdijk} (where each component is a rooted tree), but we are unsure of the precise relationship to expect.

\section{Interactions}
In this final section, we turn to some interactions between the operadic structures that have appeared so far.

\begin{theorem}\label{thm change of presheaves}
Suppose $f\colon \CC \to \DD$ is one of the following three functors:
\[ \mathbf{\Omega} \to \gcatcyc \qquad \pgcatsc \to \gcatnought \qquad \ocat \to \gcat,\]
and consider the associated adjoint string
\[
\begin{tikzcd}[column sep=large]
    \widehat{\DD}  \rar["f^*" description] 
    \rar[phantom, bend right=18, "\scriptscriptstyle\perp"]  
    \rar[phantom, bend left=18, "\scriptscriptstyle\perp"] 
    & \widehat{\CC}
    \lar[bend right=30, "f_!" swap] 
    \lar[bend left=30, "f_*"] 
\end{tikzcd}
\]
where $f^*$ is restriction along $f$, and $f_!$ and $f_*$ are left and right Kan extension.
The following hold:
\begin{enumerate}
\item If $X\in \widehat{\DD}$ is a presheaf, then $X$ is Segal if and only if $f^*X \in \widehat{\CC}$ is Segal.
\item If $Y \in \widehat{\CC}$ is Segal, then so is $f_!Y \in \widehat{\DD}$.\label{segal preservation}
\end{enumerate}
\end{theorem}
\begin{proof}
Proposition 6.17, 
Theorem 6.23, and 
Proposition 6.24 of \cite{Hackney:CGOS}.
\end{proof}

In light of the equivalence in \cref{rmk directed graphs as slice}, a directed graph is essentially just the data of an undirected graph along with an element of the orientation presheaf $\opshf$.
If $x \in \opshf_G$ where $G$ is an undirected graph, for the moment we will write $G_x$ for the associated directed graph.
If $G$ is an undirected tree and $r\in \eth(G)$, then there is a unique directed structure on $G$ which is a rooted tree with output $r$; we write $G_r$ for this rooted tree.
The main component of \cref{thm change of presheaves}\eqref{segal preservation} is the following description of the left Kan extension (see \cite[\S6.1]{Hackney:CGOS}).

\begin{proposition}\label{prop lke formula}
Let $f\colon \CC \to \DD$ be one of the forgetful functors $\pgcatsc \to \gcatnought$ or $\ocat \to \gcat$. 
If $Z \in \widehat{\CC}$ is $\CC$-presheaf, then the left Kan extension $f_!Z \in \widehat{\DD}$ is given by the formula
\[
	(f_!Z)_G = \coprod_{x\in \opshf_G} Z_{G_x}.
\]
For the functor $f\colon \dendcat \to \gcatcyc$, we have
\[
	(f_!Z)_G = \coprod_{r\in \eth(G)} Z_{G_r}.
\]
\end{proposition}

\begin{question}
Is the analogue of \cref{thm change of presheaves} also true in the $\infty$-categorical setting? (See \cref{def alg pattern}.)
That is, for these $f$, does $f_! \colon \fun(\CC^\oprm, \mathcal{S}) \to \fun(\DD^\oprm, \mathcal{S})$ restrict to a functor $\seg^\infty(\CC) \to \seg^\infty(\DD)$ between the $\infty$-categories of Segal presheaves from \cite{ChuHaugseng:HCASC}?
\end{question}

As is explained in \cite[\S4.3]{DrummondColeHackney:DKHTCO}, the category of dioperads may be viewed as a certain (non-full) subcategory of augmented cyclic operads $\diop \to \augcyc$.
It is better to describe this as a certain slice category, that is, $\diop \simeq \augcyc_{/\mathsf{IO}}$.
Here, $\mathsf{IO}$ is the terminal object of $\augcyc_{\{i,o\}}$ where the color set $\{i,o\}$ possesses the non-trivial involution.
Another way to say this is that $\mathsf{IO}$ is isomorphic to the image of the terminal dioperad under $\diop \to \augcyc$, and induces an equivalence
\[
	\diop \simeq \diop_{/\ast} \xrightarrow{\simeq} \augcyc_{/\mathsf{IO}}.
\]
The map to $\mathsf{IO}$ of course forces all cyclic operads to have free involution on color sets, and also yields a splitting of the profile of an operation into `inputs' and `outputs.'
This follows a similar approach for wheeled properads $\wproperad \simeq \modular_{/\mathsf{IO}}$ from Example 1.29 of \cite{Raynor:DLCSM}.

We use this to establish the following provisional statement (see \cref{rmk diop cyc nerves}).
\begin{proposition}\label{prop cyc diopd}
The nerve theorem for augmented cyclic operads implies that for dioperads.
\end{proposition}
\begin{proof}
Let $f\colon \pgcatsc \to \gcatnought$ be the functor forgetting the directed structure. 
Notice that $N(\mathsf{IO}) \cong \opshf$ is nothing but the orientation $\gcatnought$-presheaf from \cref{rmk directed graphs as slice}.
We then have the following commutative diagram of categories, where on presheaves we use that $f_!(\ast) = \opshf$ (by inspection, or application of \cref{prop lke formula}).
\[ \begin{tikzcd}
\diop \rar{\simeq} \dar{N} & \augcyc_{/\mathsf{IO}} \rar \dar & \augcyc \dar{N} \\
\seg(\pgcatsc) \rar \dar[hook,"\text{f.f.}"] & \seg(\gcatnought)_{/N(\mathsf{IO})} \rar\dar[hook,"\text{f.f.}"]  & \seg(\gcatnought) \dar[hook,"\text{f.f.}"] \\
\widehat{\pgcatsc} \rar{\simeq} & \widehat{\gcatnought}_{/\opshf} \rar & \widehat{\gcatnought}
\end{tikzcd} \]
The equivalence on the top left was mentioned just above, while that on the bottom left is a standard fact on slices of presheaf categories (see, for instance, Exercise 8 in Chapter III of \cite{MaclaneMoerdijk:SGL}) using the second equivalence of categories appearing in \cref{rmk directed graphs as slice}.
If $N\colon \augcyc \to \seg(\gcatnought)$ assumed to be an equivalence (that is, if the nerve theorem holds for augmented cyclic operads), then the induced map on slices $\augcyc_{/\mathsf{IO}} \to \seg(\gcatnought)_{/N(\mathsf{IO})}$ is also an equivalence.
By the left cancellation property for fully faithful functors, $N\colon \diop \to \seg(\pgcatsc)$ is fully faithful.
But this functor is also essentially surjective: if $X\in \seg(\pgcatsc)$, then we know $f_!X \to N(\mathsf{IO})$ in $\seg(\gcatnought)_{/N(\mathsf{IO})}$ is isomorphic to the image of some dioperad $D\in \diop$.
We can conclude that $X \cong N(D)$ since $\seg(\pgcatsc) \to \seg(\gcatnought)_{/N(\mathsf{IO})}$ is injective on isomorphism classes of objects.
Hence $N\colon \diop \to \seg(\gcatnought)$ is an equivalence, as desired.
\end{proof}

Essentially the same argument shows that the nerve theorem for modular operads from \cite{HRY-mod2} implies that for wheeled properads from \cite{HRYfactorizations}.

If $\CC$ is a category, write $\spshf{\CC} \coloneqq \fun(\CC^\oprm, \widehat{\simpcat})$ for the category of \emph{simplicial} presheaves.
The following question was posed by Drummond-Cole and the author \cite[Remark 6.8]{DrummondColeHackney:DKHTCO}, and has implications in the theory of cyclic 2-Segal objects \cite[Remark 5.0.20]{Walde:2SSIIO}.

\begin{question}
Let $f\colon \dendcat \to \gcatcyc$ be the functor which forgets the root.
Consider the operadic model structure on dendroidal sets $\widehat{\dendcat}$ from \cite{CisinskiMoerdijk:DSMHO} and the dendroidal Rezk model structure on $\spshf{\dendcat}$ from \cite[\S6]{CisinskiMoerdijk:DSSIO}.
We have two adjoint strings associated to the functor $f$:
\[
\begin{tikzcd}[column sep=large]
    \widehat{\gcatcyc}  \rar["f^*" description] 
    \rar[phantom, bend right=18, "\scriptscriptstyle\perp"]  
    \rar[phantom, bend left=18, "\scriptscriptstyle\perp"] 
    & \widehat{\dendcat}
    \lar[bend right=30, "f_!" swap, start anchor = {[yshift=1.5ex]west}, end anchor = {[yshift=1.5ex]east}] 
    \lar[bend left=30, "f_*", start anchor = {[yshift=-1.5ex]west}, end anchor = {[yshift=-1.5ex]east}] 
&
    \spshf{\gcatcyc}  \rar["f^*" description] 
    \rar[phantom, bend right=18, "\scriptscriptstyle\perp"]  
    \rar[phantom, bend left=18, "\scriptscriptstyle\perp"] 
    & \spshf{\dendcat}
    \lar[bend right=30, "f_!" swap, start anchor = {[yshift=1.5ex]west}, end anchor = {[yshift=1.5ex]east}] 
    \lar[bend left=30, "f_*", start anchor = {[yshift=-1.5ex]west}, end anchor = {[yshift=-1.5ex]east}] 
\end{tikzcd}
\]
Is the endo-adjunction $f^*f_! \dashv f^*f_*$ on $\widehat{\dendcat}$ (resp.\ on $\spshf{\dendcat}$) a \emph{Quillen} adjunction?
The formula from \cref{prop lke formula} may well help answer this question.
If $f^*f_!$ is a left Quillen functor, then the adjoint string lifting results from \cite{DrummondColeHackney:CERIMS} will imply the existence of a lifted model structure on $\widehat{\gcatcyc}$ (resp.\ on $\spshf{\gcatcyc}$).
Further, this should yield a Quillen equivalence with the category of simplicially-enriched cyclic operads, lifting those for simplicially-enriched operads from \cite{CisinskiMoerdijk:DSSO}.
\end{question}
For the potential relationship with cyclic 2-Segal objects mentioned in \cite{Walde:2SSIIO}, one should instead use planar and non-symmetric versions of trees and (cyclic) operads everywhere (including the model structure on planar dendroidal sets from \cite[\S8.2]{Moerdijk:LDS}).
It is likely that the planar/non-symmetric versions follow from the symmetric versions via an appropriate slicing, as in \cite{Gagna:Planar}.

\begin{question}
In \cite[Corollary 8.14]{Raynor:DLCSM}, Raynor gives a model structure
on the category of simplicial presheaves $\spshf{\mathbf{R}}$
where the fibrant objects are the projectively fibrant presheaves which satisfy a weak version of the Segal condition. 
Similarly, the category $\spshf{\gcat}$ admits Segal-type model structure, obtained by taking a left Bousfield localization of the projective model structure (a variation of this appears in \cite[Proposition 3.10]{HRY-mod1}).
It would be interesting to know if the restriction map 
\[
	\spshf{\mathbf{R}} \to \spshf{\gcat}
\]
is a right Quillen equivalence with respect to these model structures.
This would reflect the equivalence of categories $\seg(\mathbf{R}) \to \seg(\gcat)$ between ordinary simplicial presheaves (as follows from \cref{rmk KR nerves} and \cref{thm nerve}).
\end{question}

\appendix

\section{Proofs}

\subsection{Full tree maps preserve intersections}\label{subsec full tree intersections}
In this section, we show that it is automatic that full tree maps preserve intersections.
This was shown in a round-about way in \cite{Hackney:CGOS}, which went through the properadic graphical category.
Here we give a elementary direct proof, that deals only with trees.
It is important that our notion of boundary-compatibility refers to involutive maps $A_G \to A_{G'}$ rather than maps $E_G \to E_{G'}$ (compare \cite[Definition 1.12]{HRY-cyclic}).

\begin{lemma}\label{lem no overlap}
Suppose $G$ and $G'$ are undirected trees, and $(\varphi_0,\varphi_1)$ is a tree map from $G$ to $G'$.
If $u \neq v$ are distinct vertices of $G$, then the subtrees $\varphi_1(u)$ and $\varphi_1(v)$ of $G'$ do not share any vertices in common.
\end{lemma}
\begin{proof}
By way of contradiction, suppose that $w$ is a vertex in both subtrees $S = \varphi_1(u)$ and $T=\varphi_1(v)$. 
For simplicity, we first treat the case when $u$ and $v$ are adjacent vertices.
That is, there is a (unique) arc $\tilde a \in D_G$ with $t(\tilde a) = v$ and $t(\tilde a^\dagger) = u$.
Let $a = \varphi_0(\tilde a)$. 
By boundary-compatibility we have $a \in \eth(S) = \eth(\varphi_1(u))$ and $a^\dagger \in \eth(T) = \eth(\varphi_1(v))$.
Since $S$ is a tree, there is a unique path in $S$
\[
	P_S \coloneqq v_0 a_0 v_1 a_1 \cdots v_n a_n
\]
with $v_0 = w$ and $a_n = a$.
By \emph{path in $S$} we mean that $v_i \in V_S$, $a_i \in A_S$, and these elements satisfy the equations $t(a_i) = v_{i+1}$ and $t(a_i^\dagger) = v_i$ whenever both sides are defined in $S$ (see \cref{def paths trees}).
Likewise, there is a unique path in $T$
\[
	P_T \coloneqq v_0' a_0' v_1' a_1' \cdots v_m' a_m'
\]
with $v_0' = w$ and $a_m' = a^\dagger$.

Then we have a composite path in $G$
\begin{align*}
P_SP_T^\dagger & = v_0 a_0 v_1 a_1 \cdots v_n a_n v_m' (a_{m-1}')^\dagger v_{m-1}'  \cdots (a_1')^\dagger v_1' (a_0')^\dagger v_0' \\
&= v_0 a_0 v_1 a_1 \cdots v_n a_n v_{n+1} a_{n+1} v_{n+2} \cdots a_{n+m-1} v_{n+m} a_{n+m} v_{n+m+1}
\end{align*}
where $v_{n+k} = v'_{m-k+1}$ for $1 \leq k\leq m+1$ and $a_{n+j} = (a'_{m-j})^\dagger$ for $0\leq j \leq m$.
Since $v_0 = w = v_{n+m+1}$, this composite path $P_SP_T^\dagger$ is a cycle in $G'$. 
This is impossible, since $G'$ is a tree.

The general case, when $u$ and $v$ are not adjacent in $G$, is proved using similar ideas. 
Namely, there is a unique path in $G$
\[
	u_0 \tilde a_0 u_1 \tilde a_1 \dots \tilde a_\ell u_{\ell+1}
\]
with $u_0 = u$ and $u_{\ell +1} = v$.
We define $P_S$ to be the path in $S$ starting at $w$ and ending at $\varphi_0(\tilde a_0)$, and $P_T$ to be the path in $T$ starting at $w$ and ending at $\varphi_0(\tilde a_\ell^\dagger)$.
For $1 \leq j \leq \ell$ there is a unique path $P_j$ in $\varphi_1(u_j)$ which starts with $\varphi_0(\tilde a_{j-1})$ and ends with $\varphi_0(\tilde a_j)$.
The composite path $P_S P_1 \cdots P_\ell P_T^\dagger$ will be a cycle at $w$, which again is impossible since $G'$ is a tree.
\end{proof}

\begin{lemma}\label{intersection lemma}
Suppose $G$ and $G'$ are undirected trees, and $(\varphi_0,\hphi)$ is a full tree map from $G$ to $G'$.
If $S, T \in \emb(G)$ are subtrees of $G$ so that $\hphi(S)$ and $\hphi(T)$ have a vertex $w \in V_{G'}$ in common, then there is a unique vertex $v\in V_S \cap V_T \subset V_G$ so that $w\in \hphi(\medstar_v)$.
\end{lemma}
\begin{proof}
Suppose $w$ is a vertex of $\hphi(S) \cap \hphi(T)$. 
Writing $S$ as an iterated union of distinct stars
\[
	S = ((\cdots ((\medstar^1 \cup \medstar^2) \cup \medstar^3) \cup \cdots )\cup \medstar^{n-1}) \cup \medstar^n
\]
we must have $w\in \hphi(\medstar^k)$ for some $k$ since $\hphi(S)$ is the union of the $\hphi(\medstar^k)$.
Likewise, $w\in \hphi(\medstar)$ for some $\medstar \rat T$.
By \cref{lem no overlap} applied to $(\varphi_0, \hphi|_{V_G})$, we conclude that $\medstar = \medstar^k$.
Then $v$ is the vertex of the star $\medstar = \medstar^k$ which is in both $S$ and $T$.
Uniqueness is guaranteed by \cref{lem no overlap}.
\end{proof}

For subtrees of a graph $G$, the graph $S\cap T$ is either empty or again a subtree.
In particular, if $S\cap T$ does not contain a vertex, then it is either empty or a single edge.

\begin{proposition}\label{prop intersection preservation}
Full tree maps preserve intersections between overlapping subtrees.
That is, if $(\varphi_0, \hphi) \colon G \to G'$ is a full tree map and $S, T \in \emb(G)$ are subtrees with $S\cap T$ non-empty, then $\hphi(S\cap T) = \hphi(S) \cap \hphi(T)$.
\end{proposition}
\begin{proof}
Since $\hphi \colon \emb(G) \to \emb(G')$ preserves unions, it also preserves the partial order.
As $S\cap T$ is inhabited, it is a subtree of both $S$ and $T$, so 
\begin{equation}\label{eq order preservation}
\hphi(S\cap T) \subseteq \hphi(S) \cap \hphi(T).
\end{equation} 
Notice that both sides are subtrees of $G'$.
If $w$ is a vertex in $\hphi(S) \cap \hphi(T)$, then $w$ is also in $\hphi(S\cap T)$ by \cref{intersection lemma}.
Since subtrees containing a vertex are completely determined by their vertices, we conclude that \eqref{eq order preservation} is an equality.
If $\hphi(S) \cap \hphi(T)$ does not contain any vertices, then it is a single edge, so again \eqref{eq order preservation} is an equality.
\end{proof}

\subsection{Tree maps and full tree maps coincide}\label{sec trees vs full trees}
Our goal is to prove that there is no real difference between tree maps and full tree maps from \cref{def tree map}.

\begin{lemma}\label{lem subtree union boundary}
Suppose $S$ and $T$ are subtrees of a tree $G$, and that $S$ and $T$ intersect at a single edge $e = [a, a^\dagger]$, where $a \in \eth(S)$ and $a^\dagger \in \eth(T)$.
If neither $S$ nor $T$ is an edge, then
\[
	\eth(S \cup T) = (\eth(S) \backslash \{a\}) \amalg (\eth(T) \backslash \{a^\dagger\}).
\]
If $S$ is an edge then $\eth(S\cup T) = \eth(T)$ and if $T$ is an edge then $\eth(S\cup T) = \eth(S)$.
\end{lemma}
\begin{proof}
If $S$ is an edge then $S\cup T = T$, and similarly $S\cup T = S$ when $T$ is an edge, so the final statement is immediate.
We suppose that neither $S$ nor $T$ is an edge and compute the boundary of the union.

Since $S$ is not an edge and $a\in \eth(S)$, there is a vertex $v\in V_S$ with $t(a^\dagger) = v$.
Since $v$ is also in $V_{S\cup T}$, it follows that $a^\dagger \notin \eth(S\cup T)$.
A symmetric argument shows that $a\notin \eth(S\cup T)$.

Suppose $a' \in \eth(S) \backslash \{a\}$, so that $t((a')^\dagger) \in V_S$. 
We either have $a' \in \eth(G)$, or $t(a') = v$ is a vertex of $G$ which is not in $S$. 
If $a' \in \eth(G)$, then $a' \in \eth(S\cup T)$ as well.
If $t(a') = v \in V_G \backslash V_S$, then either $v \in V_T$ or $v\notin V_T$.
If $v\in V_T$, then $e' = [a', (a')^\dagger]$ is an edge in $S\cap T$ distinct from $e = [a,a^\dagger]$, contrary to our assumption. 
If $v\notin V_T$, then $v\notin V_{S\cup T}$, hence $a' \in \eth(S\cup T)$. 
Since $a'$ was arbitrary, we have shown $\eth(S) \backslash \{a\} \subset \eth(S\cup T)$; reversing roles gives $\eth(T) \backslash\{a^\dagger\} \subset \eth(S\cup T)$.

It remains to observe that $\eth(S\cup T) \subset \eth(S) \cup \eth(T)$.
Suppose $a'\in \eth(S\cup T)$; since $S\cup T$ is not an edge, we have $t((a')^\dagger) \in V_{S\cup T} = V_S \cup V_T$. 
Without loss of generality suppose $t((a')^\dagger) \in V_S$.
Since $a'$ is in  $\eth(S\cup T)$, either $a'\in \eth(G)$ or $a'\in D_G$ and $t(a') = w \in V_G \backslash V_{S\cup T} \subseteq V_G \backslash V_S$.
In both cases, we conclude that $a'$ is in $\eth(S)$, as desired.
\end{proof}

\begin{lemma}\label{lemma avoid subtree}
Suppose $G$ is a tree and $T$ is a proper subtree.
Then there exists a star subtree $\medstar \subset G$ and a proper subtree $R \subset G$ so that
$R$ overlaps with $\medstar$,
the union $R \cup \medstar$ is equal to $G$, and
$T \subset R$.
\end{lemma}
\begin{proof}
Note that an edge does not have proper subtrees, so $G$ has at least one vertex.
If $G$ is a star, then $T$ must be an edge, and one takes $\medstar = G$ and $R = T$.
If $G$ has more than one vertex and $T$ is an edge, then one can take $\medstar$ to be any extremal vertex of $G$ which does not contain $T$ as a boundary, and $R$ to be the complement of $\medstar$.

We now suppose $T$ contains a vertex.
Let $v \in V_G\backslash V_T$ be a vertex of minimum positive distance from $T$, that is, there is an edge $e$ between $v$ and some vertex $w$ of $T$.
Consider the collection of all paths of $G$ which contain $e$ but do not contain $w$, and let $S\subset G$ be the subtree consisting of all vertices appearing on such paths, and all arcs appearing on edges on such paths. 
Notice that $v\in V_S$.
If $S$ is a star then $v$ is extremal in $G$, and we set $\medstar = S$ and let $R$ be its complement.
Otherwise, let $u \in V_S$ be an extremal vertex of $S$ other than $v$ (which may or may not be extremal in $S$).
The vertex $u$ is also extremal in $G$, and we let $\medstar$ be its star and $R$ the complement.
In either case, the three conditions hold.
\end{proof}

\begin{proposition}\label{prop tree maps equiv}
Suppose $G$ and $G'$ are undirected trees.
The assignment $(\varphi_0,\hphi) \mapsto (\varphi_0,\hphi|_{V_G})$ is a bijection between tree maps and full tree maps from \cref{def tree map}. 
\end{proposition}
\begin{proof}
Let $\emb^n(G) \subseteq \emb(G)$ denote set of subtrees having $n$ or fewer vertices.
This gives a (finite) filtration of the set $\emb(G)$, starting with \[ \emb^0(G) = E_G \subseteq E_G \amalg V_G = \emb^1(G).\]
Given a tree map $(\varphi_0, \varphi_1)$, we inductively define $\hphi_n \colon \emb^n(G) \to \emb(G')$ and show that it satisfies boundary-compatibility and preserves unions that exist in $\emb^n(G)$.
The first map $\hphi_0 \colon E_G \to E_{G'} \subseteq \emb(G')$ is induced from $\varphi_0 \colon A_G \to A_{G'}$, while $\hphi_1$ is defined to be $\hphi_0 \amalg \varphi_1$.
These both satisfy boundary-compatibility.
Any unions $S\cup T \in \emb^1(G)$ for overlapping $S,T$ which exist must either have $S=T$ or one of the two terms is an edge.
If $S$ is an edge or if $S = T$, then $S\cup T = T$, and also $\hphi_1(S) \cup \hphi_1(T) = \hphi_1(T)$ since either $\hphi_1(S)$ is an edge or $\hphi_1(S) = \hphi_1(T)$.
Thus $\hphi_1$ preserves unions.

We continue inductively.
Suppose $\hphi_n \colon \emb^n(G) \to \emb(G)$ is defined, preserves unions that exist in the $n$th filtration level, and $(\varphi_0 ,\hphi_n)$ is boundary-compatible.
Let $n\geq 1$.
Any subtree $R$ having $n+1$ vertices can be written as $S \cup \medstar$ where $\medstar$ is a star and $S$ has $n$ vertices. 
It is automatic that $S$ and $\medstar$ intersect only at an edge $e = [a,a^\dagger]$.
We make a provisional definition $\hphi_{n+1}(R) \coloneqq \hphi_n(S) \cup \hphi_1(\medstar)$, which of course depends on a choice of $S \subset R$ having $n$ vertices.
If $n+1 = 2$, then $S$ is itself a star and by symmetry there is no actual choice involved in this definition.
Given a different subtree $T \subset R$ having $n \geq 2$ vertices, the tree $S \cap T$ will have $n-1$ vertices.
We have $R = T \cup \medstar'$ for a star $\medstar' \subset S \subset R$, with $T$ and $\medstar'$ intersecting at an edge $e'$.
Note that $e, e'$ are both in $S\cap T$.
We then have 
\begin{align*}
\hphi_n(T) \cup \hphi_1(\medstar') &= \hphi_n((S\cap T) \cup \medstar) \cup \hphi_1(\medstar') \\ 
&= \hphi_{n-1}(S\cap T) \cup \hphi_1(\medstar) \cup \hphi_1(\medstar') \\ 
&= \hphi_{n}((S\cap T) \cup \medstar') \cup \hphi_1(\medstar) = \hphi_n (S) \cup \hphi_1(\medstar)
\end{align*}
which shows that $\hphi_{n+1}$ is well-defined.

Further, observe that $\hphi_n(S)$ and $\hphi_1(\medstar)$ (as above) intersect only at the single edge $\bar e = [\varphi_0(a), \varphi_0(a)^\dagger]$.
Indeed, the proof of \cref{intersection lemma} works equally well for $\hphi_n$ as it does for $\hphi$, so $\hphi_n(S)$ and $\hphi_n(\medstar)$ cannot share any vertices.

We can use this to show that $\hphi_{n+1}$ is boundary-compatible.
As above, suppose that $R = S \cup \medstar$ has $n+1$ vertices, $\medstar$ is a star, $S$ has $n$ vertices, and $S$ and $\medstar$ intersect only at the edge $e = [a,a^\dagger]$.
Here we take $a\in \eth(S)$ and $a^\dagger \in \eth(\medstar)$; write $\bar a = \varphi_0(a)$.
Boundary-compatibility implies that 
\[ \begin{tikzcd}
\eth(S) \rar{\varphi_0} & \eth(\hphi_n(S)) & \eth(\medstar) \rar{\varphi_0} & \eth(\hphi_1(\medstar))
\end{tikzcd} \]
are bijections.
If neither $\hphi_n(S)$ nor $\hphi_1(\medstar)$ is an edge, then
we have the desired bijection below by \cref{lem subtree union boundary}.
\[ \begin{tikzcd}[column sep=small]
\eth(S \cup \medstar) \ar[rr,leftrightarrow,"="] & & (\eth(S) \backslash \{a\}) \amalg (\eth(\medstar) \backslash \{a^\dagger\}) \dar["\cong"] \\
\eth(\hphi_{n+1}(S\cup \medstar)) \rar[leftrightarrow,"="] & 
\eth(\hphi_n(S) \cup \hphi_1(\medstar)) 
\rar[leftrightarrow,"="] & 
(\eth(\hphi_n(S)) \backslash \{\bar a\}) \amalg (\eth(\hphi_1(\medstar)) \backslash \{\bar a^\dagger\})
\end{tikzcd} \]
We must still address the case when $\hphi_n(S)$ is an edge or when $\hphi_1(\medstar)$ is an edge.
These cases are entirely analogous, so we suppose that $\hphi_1(\medstar)$ is an edge.
By boundary-compatibility, $\eth(\medstar) = \{a^\dagger, a'\}$ with $\varphi_0((a')^\dagger) = \varphi_0(a^\dagger) = \bar a^\dagger$, so $\varphi_0(a') = \bar a$.
Thus
\[ 
\eth(S \cup \medstar) = (\eth(S) \backslash \{a\}) \amalg (\eth(\medstar) \backslash \{a^\dagger\})   =  (\eth(S) \backslash \{a\})  \amalg \{a'\} \to \eth( \hphi_n(S) )
\]
is a bijection, as required.
We have thus shown that $(\varphi_0, \hphi_{n+1})$ is boundary-compatible.

Finally, we must show that $\hphi_{n+1}$ preserves unions that exist in $\emb^{n+1}(G)$.
Suppose that $S \cup T \in \emb^{n+1}(G)$ is a union of overlapping subtrees $S$ and $T$. 
We may assume that $S\cup T$ has $n+1$ vertices, for if it has $n$ or fewer vertices then $\hphi_{n+1}(S\cup T) = \hphi_n(S\cup T) = \hphi_n(S) \cup \hphi_n(T)$ by the induction hypothesis.
We further assume that $S \neq T$, since if $S = T$ then $\hphi_{n+1}(S) \cup \hphi_{n+1}(T) = \hphi_{n+1}(S) = \hphi_{n+1}(S \cup T)$.
Without loss of generality, suppose there is a vertex of $S$ which is not in $T$.
By \cref{lemma avoid subtree}, there is a decomposition $S\cup T = R \cup \medstar$ where $T$ is a subtree of $R$, and $R$ is a proper subtree of $S\cup T$.
Notice also that $\medstar \subset S$.
Of course $S\cap R$ is a subtree since it is non-empty (it contains $S\cap T$).
We then have $S = S \cap (S\cup T) = S \cap (\medstar \cup R) = \medstar \cup (S\cap R)$.
Now $S\cap R$ overlaps with $T$, so 
\[
	S \cup T = ( \medstar \cup (S\cap R) ) \cup T = \medstar \cup ((S\cap R) \cup T).
\]
The union $(S\cap R) \cup T$ contains exactly $n$ vertices, and we find
\begin{align*}
\hphi_{n+1}(S\cup T) &= \hphi_1(\medstar) \cup \hphi_n((S\cap R) \cup T) \\ 
&= \hphi_1(\medstar) \cup (\hphi_n(S\cap R) \cup \hphi_n(T) ) \\
&= ( \hphi_1(\medstar) \cup \hphi_n(S\cap R) ) \cup \hphi_n(T) \\
&= \hphi_{n+1}(\medstar \cup (S\cap R)) \cup \hphi_n(T) = \hphi_{n+1}(S) \cup \hphi_{n+1}(T).
\end{align*}
This uses that $\hphi_n$ preserves unions and the definition of $\hphi_{n+1}$.

If $G$ has $k$ vertices, then $\emb^k(G) = \emb(G)$ and we have thus defined a full tree map from $G$ to $G'$.
This construction is a section to the assignment from the proposition statement.
But notice that it is the \emph{unique} such section: our $\hphi$ must preserve unions, and we defined the section so that it would preserve particular unions.
Thus the map in question is a bijection.
\end{proof}

\subsection{The (augmented) cyclic operad associated to a tree}\label{subsec tree maps and cyclic maps}
Our goal in this section is to prove \cref{prop gcatnought ff}.
The reader should recall the category of augmented cyclic operads from Definition 2.3 of \cite{DrummondColeHackney:DKHTCO}, which we have denoted by $\augcyc$.
If $G$ is an undirected tree, then the (augmented) cyclic operad $\freecyc(G) \in \augcyc_{A_G}$ is obtained by choosing for each $v\in V_G$ an ordering $a_0, a_1, \dots, a_n$ of $\eth(\medstar_v) = \nbhd(v)^\dagger$, and considering $v$ as a free generator in $\freecyc(G)(a_0, a_1, \dots, a_n)$.
This choice of order does not matter for the isomorphism class of $\freecyc(G)$.

\begin{lemma}\label{lem operations are subtrees}
The operations in $\freecyc(G)$ are in bijection with subtrees $T$ of $G$ equipped with a total ordering on the boundary $\eth(T)$.
\end{lemma}

\begin{proof}
Let $P= \freecyc(G)$. 
We have identity elements $\id_a \in P(a^\dagger, a)$ and $\id_{a^\dagger} \in P(a, a^\dagger)$, and these correspond to the two possible orderings for the boundary of the edge subtree $[a,a^\dagger]$.

Working inductively by composing with one generator at a time, each non-edge subtree $T$ determines an element of $P(\eth(T))$ for some ordering of $\eth(T)$ by \cref{lem subtree union boundary}.
The axioms of a cyclic operad guarantee that this element does not depend on the choice of order of attaching vertices to construct the tree $T$.
This is an element $T\in P(\eth(T)) = P(a_0, \dots, a_n)$, where the ordering of $\eth(T)$ is determined by the chosen orderings on $\eth(\medstar_v)$.
The cyclic operad axioms imply that applying a permutation at any intermediate step simply has the effect of applying some permutation at the last step
\[
	T \in P(a_0, a_1, \dots, a_n) \xrightarrow{\sigma^*} P(a_{\sigma(0)}, a_{\sigma(1)}, \dots, a_{\sigma(n)}).
\]
Thus all subtrees of equipped with an ordering of the boundary may regarded as (necessarily distinct, by \cite[Lemma A.3]{Hackney:CGOS}) operations of $P$; we refer to these as simply subtrees in the next paragraph.

Now suppose that $S,T$ are two non-edge subtrees of $G$,
considered as elements 
\[
	S \in P(a'_0, \dots, a'_m) = P(\eth(S)) \qquad T \in P(a_0, \dots, a_n) = P(\eth(T))
\]
with $a'_i = a_j^\dagger$, so that we can form the composition $S {\vphantom{\circ}_{i}\circ_{j}} T$.
The paths argument in the first two paragraphs of the proof of \cref{lem no overlap} shows that $S$ and $T$ do not share any vertices, hence the composite of $S$ and $T$ uses each generator at of $P$ at most once.
It follows that the set of generators used is $V_S \amalg V_T = V_{S\cup T}$, and the composite is just the union $S\cup T$, along with the relevant ordering of $\eth(S\cup T)$.
Since composites of subtrees are again subtrees, all iterated composites of the generators of $P$ are subtrees. 
\end{proof}

\begin{proof}[Proof of \cref{prop gcatnought ff}]
If $H$ and $G$ are undirected trees, then a map of cyclic operads $F\colon \freecyc(H) \to \freecyc(G)$ consists of an involutive function $f_0\colon A_H \to A_G$, and a function $f$ from $V_H$ to the operations of $\freecyc(G)$.
As we are regarding $v\in V_H$ as living in $\freecyc(H)(\eth(\medstar_v)) = \freecyc(H)(a_0, \dots, a_n)$, the pair of functions needs only satisfy color-compatibility, that is, $f(v)$ should live in $\freecyc(G)(f_0(a_0), \dots, f_0(a_n))$.
By \cref{lem operations are subtrees}, $f(v)$ is a subtree $T$ of $G$ together with an ordering of its boundary $\eth(T) = \{f_0(a_0), \dots, f_0(a_n)\}$.
But this ordering is not actual data of the morphism, rather is just serves to establish boundary-compatibility.
Thus, this assignment is exactly the same thing as a tree map of \cref{def tree map}.
Since $\freecyc$ is a functor, this establishes the desired equivalence.
\Cref{lem operations are subtrees} also implies that $\freecyc(G)$ is a cyclic operad (rather than an augmented cyclic operad) if and only if $G$ has inhabited boundary, giving the statement about $\gcatcyc$.
\end{proof} 

\noindent\textbf{Acknowledgments.}
I am grateful to the organizers of the special session for inviting me to present this work, and to the other speakers and participants for helping to make an invigorating meeting.

\providecommand{\bysame}{\leavevmode\hbox to3em{\hrulefill}\thinspace}
\providecommand{\MR}{\relax\ifhmode\unskip\space\fi MR }
\providecommand{\MRhref}[2]{%
  \href{http://www.ams.org/mathscinet-getitem?mr=#1}{#2}
}
\providecommand{\href}[2]{#2}


\begin{thebibliography}{BGMS21}

\bibitem[BB17]{BataninBerger:HTAPM}
M.~A. Batanin and C.~Berger, \emph{Homotopy theory for algebras over polynomial
  monads}, Theory Appl. Categ. \textbf{32} (2017), Paper No. 6, 148--253.
  \MR{3607212}

\bibitem[Ber]{Berger:MCO}
Clemens Berger, \emph{Moment categories and operads}, Preprint,
  \href{https://arxiv.org/abs/2102.00634v2}{arXiv:2102.00634v2} [math.CT].

\bibitem[BG19]{BourkeGarner:MT}
John Bourke and Richard Garner, \emph{Monads and theories}, Adv. Math.
  \textbf{351} (2019), 1024--1071. \MR{3956765}

\bibitem[BGMS21]{BaezGenoveseMasterShulman}
John~C. Baez, Fabrizio Genovese, Jade Master, and Michael Shulman,
  \emph{Categories of nets}, 2021 36th {A}nnual {ACM}/{IEEE} {S}ymposium on
  {L}ogic in {C}omputer {S}cience ({LICS}), IEEE, NJ, 2021, pp.~1--13.
  \MR{4401577}

\bibitem[BK]{BergerKaufmann:TGA}
Clemens Berger and Ralph~M. Kaufmann, \emph{Trees, graphs and aggregates: a
  categorical perspective on combinatorial surface topology, geometry, and
  algebra}, Preprint,
  \href{https://arxiv.org/abs/2201.10537v1}{arXiv:2201.10537v1} [math.AT].

\bibitem[BM96]{BehrendManin:SSMGWI}
K.~Behrend and Yu. Manin, \emph{Stacks of stable maps and {G}romov-{W}itten
  invariants}, Duke Math. J. \textbf{85} (1996), no.~1, 1--60. \MR{1412436}

\bibitem[BM11]{bm_reedy}
Clemens Berger and Ieke Moerdijk, \emph{On an extension of the notion of
  {R}eedy category}, Math. Z. \textbf{269} (2011), no.~3-4, 977--1004.
  \MR{2860274}

\bibitem[BM15]{BataninMarkl:OCDDC}
Michael Batanin and Martin Markl, \emph{Operadic categories and duoidal
  {D}eligne's conjecture}, Adv. Math. \textbf{285} (2015), 1630--1687.
  \MR{3406537}

\bibitem[BMW12]{BergerMelliesWeber:MAAT}
Clemens Berger, Paul-Andr\'{e} Melli\`es, and Mark Weber, \emph{Monads with
  arities and their associated theories}, J. Pure Appl. Algebra \textbf{216}
  (2012), no.~8-9, 2029--2048. \MR{2925893}

\bibitem[BR]{BorghiRobertson:MIOGT}
Olivia Borghi and Marcy Robertson, \emph{Lecture notes on modular infinity
  operads and {G}rothendieck--{T}eichm\"uller theory}, to appear.

\bibitem[CGR14]{ChengGurskiRiehl:CMMAM}
Eugenia Cheng, Nick Gurski, and Emily Riehl, \emph{Cyclic multicategories,
  multivariable adjunctions and mates}, J. K-Theory \textbf{13} (2014), no.~2,
  337--396. \MR{3189430}

\bibitem[CH20]{ChuHaugseng:EIO}
Hongyi Chu and Rune Haugseng, \emph{Enriched {$\infty$}-operads}, Adv. Math.
  \textbf{361} (2020), 106913. \MR{4038556}

\bibitem[CH21]{ChuHaugseng:HCASC}
\bysame, \emph{Homotopy-coherent algebra via {S}egal conditions}, Adv. Math.
  \textbf{385} (2021), 107733. \MR{4256131}

\bibitem[CH22]{ChuHackney}
Hongyi Chu and Philip Hackney, \emph{On rectification and enrichment of
  infinity properads}, J. Lond. Math. Soc. (2) \textbf{105} (2022), no.~1,
  1418--1517. \MR{4389306}

\bibitem[CM11]{CisinskiMoerdijk:DSMHO}
Denis-Charles Cisinski and Ieke Moerdijk, \emph{Dendroidal sets as models for
  homotopy operads}, J. Topol. \textbf{4} (2011), no.~2, 257--299. \MR{2805991}

\bibitem[CM13a]{CisinskiMoerdijk:DSSIO}
\bysame, \emph{Dendroidal {S}egal spaces and {$\infty$}-operads}, J. Topol.
  \textbf{6} (2013), no.~3, 675--704. \MR{3100887}

\bibitem[CM13b]{CisinskiMoerdijk:DSSO}
\bysame, \emph{Dendroidal sets and simplicial operads}, J. Topol. \textbf{6}
  (2013), no.~3, 705--756.

\bibitem[DCH19]{DrummondColeHackney:CERIMS}
Gabriel~C. Drummond-Cole and Philip Hackney, \emph{A criterion for existence of
  right-induced model structures}, Bull. Lond. Math. Soc. \textbf{51} (2019),
  no.~2, 309--326. \MR{3937590}

\bibitem[DCH21]{DrummondColeHackney:DKHTCO}
\bysame, \emph{Dwyer--{K}an homotopy theory for cyclic operads}, Proc. Edinb.
  Math. Soc. (2) \textbf{64} (2021), no.~1, 29--58. \MR{4249838}

\bibitem[Dun06]{Duncan:TQC}
Ross Duncan, \emph{Types for quantum computing}, Ph.D. thesis, 2006, Oxford
  University.

\bibitem[FBSD21]{FBSD:ComplexSystem}
John~D. Foley, Spencer Breiner, Eswaran Subrahmanian, and John~M. Dusel,
  \emph{Operads for complex system design specification, analysis and
  synthesis}, Proc. R. Soc. A. \textbf{477} (2021), no.~2250, Paper No.
  20210099, 35. \MR{4291809}

\bibitem[Gag15]{Gagna:Planar}
Andrea Gagna, \emph{The {C}isinski--{M}oerdijk model structure on planar
  dendroidal sets}, 2015, Master thesis, Universiteit Leiden.

\bibitem[Gan03]{Gan:KDD}
Wee~Liang Gan, \emph{Koszul duality for dioperads}, Math. Res. Lett.
  \textbf{10} (2003), no.~1, 109--124. \MR{1960128}

\bibitem[Gar08]{Garner:PPDL}
Richard Garner, \emph{Polycategories via pseudo-distributive laws}, Adv. Math.
  \textbf{218} (2008), no.~3, 781--827. \MR{2414322}

\bibitem[Get09]{Getzler:OR}
Ezra Getzler, \emph{Operads revisited}, Algebra, arithmetic, and geometry: in
  honor of {Y}u. {I}. {M}anin. {V}ol. {I}, Progr. Math., vol. 269,
  Birkh\"{a}user Boston, Boston, MA, 2009, pp.~675--698. \MR{2641184}

\bibitem[GH18]{GarnerHirschowitz:SMAF}
Richard Garner and Tom Hirschowitz, \emph{Shapely monads and analytic
  functors}, J. Logic Comput. \textbf{28} (2018), no.~1, 33--83. \MR{3798968}

\bibitem[GK95]{GetzlerKapranov:COCH}
E.~Getzler and M.~M. Kapranov, \emph{Cyclic operads and cyclic homology},
  Geometry, topology, \& physics, Conf. Proc. Lecture Notes Geom. Topology, IV,
  Int. Press, Cambridge, MA, 1995, pp.~167--201. \MR{1358617}

\bibitem[GK98]{GetzlerKapranov:MO}
\bysame, \emph{Modular operads}, Compositio Math. \textbf{110} (1998), no.~1,
  65--126. \MR{1601666}

\bibitem[Gro61]{Grothendieck:TCTGA3}
Alexander Grothendieck, \emph{Techniques de construction et th\'eor\`emes
  d'existence en g\'eom\'etrie alg\'ebrique {III} : pr\'esch\'emas quotients},
  S\'eminaire Bourbaki : ann\'ees 1960/61, expos\'es 205-222, S\'eminaire
  Bourbaki, no.~6, Soci\'et\'e math\'ematique de France, 1961, talk:212 (fr).
  \MR{1611786}

\bibitem[Hac]{Hackney:CGOS}
Philip Hackney, \emph{Categories of graphs for operadic structures}, Preprint,
  \href{https://arxiv.org/abs/2109.06231v1}{arXiv:2109.06231v1} [math.CT].

\bibitem[HHM16]{HeutsHinichMoerdijk}
Gijs Heuts, Vladimir Hinich, and Ieke Moerdijk, \emph{On the equivalence
  between {L}urie's model and the dendroidal model for infinity-operads}, Adv.
  Math. \textbf{302} (2016), 869--1043. \MR{3545944}

\bibitem[HK]{HaugsengKock}
Rune Haugseng and Joachim Kock, \emph{$\infty$-operads as symmetric monoidal
  $\infty$-categories}, Preprint,
  \href{https://arxiv.org/abs/2106.12975}{arXiv:2106.12975} [math.CT].

\bibitem[HM22]{HeutsMoerdijk:SDHT}
Gijs Heuts and Ieke Moerdijk, \emph{Simplicial and dendroidal homotopy theory},
  Ergebnisse der Mathematik und ihrer Grenzgebiete. 3. Folge. A Series of
  Modern Surveys in Mathematics, Cham: Springer, 2022 (English).

\bibitem[HR15]{HackneyRobertson:OCP}
Philip Hackney and Marcy Robertson, \emph{On the category of props}, Appl.
  Categ. Structures \textbf{23} (2015), no.~4, 543--573. \MR{3367131}

\bibitem[HR18]{MR3792530}
\bysame, \emph{Lecture notes on infinity-properads}, 2016 {MATRIX} annals,
  MATRIX Book Ser., vol.~1, Springer, Cham, 2018, pp.~351--374. \MR{3792530}

\bibitem[HRY15]{HRYbook}
Philip Hackney, Marcy Robertson, and Donald Yau, \emph{Infinity properads and
  infinity wheeled properads}, Lecture Notes in Mathematics, vol. 2147,
  Springer, Cham, 2015. \MR{3408444}

\bibitem[HRY18]{HRYfactorizations}
\bysame, \emph{On factorizations of graphical maps}, Homology Homotopy Appl.
  \textbf{20} (2018), no.~2, 217--238. \MR{3812464}

\bibitem[HRY19]{HRY-cyclic}
\bysame, \emph{Higher cyclic operads}, Algebr. Geom. Topol. \textbf{19} (2019),
  no.~2, 863--940. \MR{3924179}

\bibitem[HRY20a]{HRY-mod1}
\bysame, \emph{A graphical category for higher modular operads}, Adv. Math.
  \textbf{365} (2020), 107044, 61. \MR{4064770}

\bibitem[HRY20b]{HRY-mod2}
\bysame, \emph{Modular operads and the nerve theorem}, Adv. Math. \textbf{370}
  (2020), 107206, 39. \MR{4099828}

\bibitem[HV02]{HinichVaintrob:COACD}
Vladimir Hinich and Arkady Vaintrob, \emph{Cyclic operads and algebra of chord
  diagrams}, Selecta Math. (N.S.) \textbf{8} (2002), no.~2, 237--282.
  \MR{1913297}

\bibitem[JK11]{JoyalKock:FGNTCSM}
Andr\'e Joyal and Joachim Kock, \emph{Feynman graphs, and nerve theorem for
  compact symmetric multicategories (extended abstract)}, Electronic Notes in
  Theoretical Computer Science \textbf{270} (2011), no.~2, 105--113.

\bibitem[KL17]{KaufmannLucas:DFC}
Ralph Kaufmann and Jason Lucas, \emph{Decorated {F}eynman categories}, J.
  Noncommut. Geom. \textbf{11} (2017), no.~4, 1437--1464. \MR{3743229}

\bibitem[KM94]{KontsevichManin:GWCQCEG}
M.~Kontsevich and Yu. Manin, \emph{Gromov-{W}itten classes, quantum cohomology,
  and enumerative geometry}, Comm. Math. Phys. \textbf{164} (1994), no.~3,
  525--562. \MR{1291244}

\bibitem[Koc11]{Kock:PFT}
Joachim Kock, \emph{Polynomial functors and trees}, Int. Math. Res. Not. IMRN
  (2011), no.~3, 609--673. \MR{2764874}

\bibitem[Koc16]{Kock:GHP}
\bysame, \emph{Graphs, hypergraphs, and properads}, Collect. Math. \textbf{67}
  (2016), no.~2, 155--190. \MR{3484016}

\bibitem[Koc22]{Kock:WGPNP}
\bysame, \emph{Whole-grain {P}etri nets and processes}, J. ACM (2022),
  \href{https://doi.org/10.1145/3559103}{doi:10.1145/3559103}.

\bibitem[KW17]{KaufmannWard:FC}
Ralph~M. Kaufmann and Benjamin~C. Ward, \emph{Feynman categories},
  Ast\'{e}risque (2017), no.~387, vii+161. \MR{3636409}

\bibitem[Lur]{Lurie:HA}
Jacob Lurie, \emph{Higher algebra}, manuscript available at
  \url{https://www.math.ias.edu/~lurie/papers/HA.pdf}.

\bibitem[Mac65]{MacLane:CA}
Saunders Mac~Lane, \emph{Categorical algebra}, Bull. Amer. Math. Soc.
  \textbf{71} (1965), 40--106. \MR{171826}

\bibitem[MM92]{MaclaneMoerdijk:SGL}
Saunders Mac~Lane and Ieke Moerdijk, \emph{Sheaves in geometry and logic: a
  first introduction to topos theory}, Universitext, Springer-Verlag, New-York,
  1992. \MR{1300636}

\bibitem[MMS09]{MarklMerkulovShadrin}
M.~Markl, S.~Merkulov, and S.~Shadrin, \emph{Wheeled {PROP}s, graph complexes
  and the master equation}, J. Pure Appl. Algebra \textbf{213} (2009), no.~4,
  496--535. \MR{2483835}

\bibitem[Moe10]{Moerdijk:LDS}
Ieke Moerdijk, \emph{Lectures on dendroidal sets}, Simplicial methods for
  operads and algebraic geometry, Adv. Courses Math. CRM Barcelona,
  Birkh\"auser/Springer Basel AG, Basel, 2010, Notes written by Javier J.
  Guti\'errez, pp.~1--118. \MR{2778589}

\bibitem[MW07]{MoerdijkWeiss:DS}
Ieke Moerdijk and Ittay Weiss, \emph{Dendroidal sets}, Algebr. Geom. Topol.
  \textbf{7} (2007), 1441--1470. \MR{2366165 (2009d:55014)}

\bibitem[Ray21]{Raynor:DLCSM}
Sophie Raynor, \emph{Graphical combinatorics and a distributive law for modular
  operads}, Adv. Math. \textbf{392} (2021), Paper No. 108011, 87. \MR{4316667}

\bibitem[Rez01]{Rezk:MHTHT}
Charles Rezk, \emph{A model for the homotopy theory of homotopy theory}, Trans.
  Amer. Math. Soc. \textbf{353} (2001), no.~3, 973--1007. \MR{1804411}

\bibitem[Shu20]{Shulman:2Chu}
Michael Shulman, \emph{The 2-{C}hu-{D}ialectica construction and the
  polycategory of multivariable adjunctions}, Theory Appl. Categ. \textbf{35}
  (2020), Paper No. 4, 89--136. \MR{4063570}

\bibitem[Spi13]{Spivak:OWD}
David~I. Spivak, \emph{The operad of wiring diagrams: formalizing a graphical
  language for databases, recursion, and plug-and-play circuits}, Preprint,
  \href{https://arxiv.org/abs/1305.0297}{arXiv:1305.0297} [cs.DB], 2013.

\bibitem[Val07]{Vallette:KDP}
Bruno Vallette, \emph{A {K}oszul duality for {PROP}s}, Trans. Amer. Math. Soc.
  \textbf{359} (2007), no.~10, 4865--4943. \MR{2320654 (2008e:18020)}

\bibitem[VSL15]{VagnerSpivakLerman:AODSOWD}
Dmitry Vagner, David~I. Spivak, and Eugene Lerman, \emph{Algebras of open
  dynamical systems on the operad of wiring diagrams}, Theory Appl. Categ.
  \textbf{30} (2015), Paper No. 51, 1793--1822. \MR{3431301}

\bibitem[Wal21]{Walde:2SSIIO}
Tashi Walde, \emph{2-{S}egal spaces as invertible infinity-operads}, Algebr.
  Geom. Topol. \textbf{21} (2021), no.~1, 211--246. \MR{4224740}

\bibitem[Web07]{Weber:F2FPRA}
Mark Weber, \emph{Familial 2-functors and parametric right adjoints}, Theory
  Appl. Categ. \textbf{18} (2007), No. 22, 665--732. \MR{2369114}

\bibitem[Wei11]{Weiss:FODS}
Ittay Weiss, \emph{From operads to dendroidal sets}, Mathematical foundations
  of quantum field theory and perturbative string theory, Proc. Sympos. Pure
  Math., vol.~83, Amer. Math. Soc., Providence, RI, 2011, pp.~31--70.
  \MR{2742425}

\bibitem[YJ15]{YauJohnson:FPAM}
Donald Yau and Mark~W. Johnson, \emph{A foundation for {PROP}s, algebras, and
  modules}, Mathematical Surveys and Monographs, vol. 203, American
  Mathematical Society, Providence, RI, 2015. \MR{3329226}

\end{thebibliography}
\end{document}